\documentclass[11pt]{amsart}
\usepackage{amsfonts}
\usepackage{mathrsfs}
\usepackage{graphicx}
\usepackage{verbatim}
\usepackage[abs]{overpic}
\usepackage{amsmath,amsthm,amssymb}
\usepackage{geometry}
\newtheorem{theorem}{Theorem}[section]
\newtheorem{lemma}[theorem]{Lemma}
\newtheorem{corollary}[theorem]{Corollary}

\theoremstyle{definition}
\newtheorem{definition}[theorem]{Definition}

\theoremstyle{remark}
\newtheorem{remark}[theorem]{Remark}
\numberwithin{equation}{section}

\begin{document}
\title[Uniqueness of Stein fillings]{Stein fillings of contact 3-manifolds\\
 obtained as Legendrian surgeries}


\author{Amey Kaloti }
\address{School of Mathematics \\ Georgia Institute of Technology\\Atlanta, GA 30332, USA}
\email{ameyk@math.gatech.edu}



\author { Youlin Li}
\address{Department of Mathematics \\Shanghai Jiao Tong University\\Shanghai 200240, China}
\email{liyoulin@sjtu.edu.cn}





\begin{abstract}
In this paper, we classify Stein fillings of an infinite family of contact 3-manifolds up to diffeomorphism.  Some contact 3-manifolds in this family  can be obtained by Legendrian surgeries on $(S^3,\xi_{std})$ along certain Legendrian 2-bridge knots. We also classify Stein fillings,  up to symplectic deformation, of an infinite family of contact 3-manifolds which can be obtained by Legendrian surgeries on $(S^3,\xi_{std})$ along certain Legendrian twist knots.  As a corollary, we obtain a classification of Stein fillings of an infinite family of contact hyperbolic 3-manifolds  up to symplectic deformation.
\end{abstract}

\maketitle

\section{\textbf{Introduction}}

A \textit{Stein manifold} is a complex manifold that admits a proper holomorphic embedding into $\mathbb{C}^{N}$ for some large integer $N$. According to \cite{Grauert_Levi's_problem_and_imbedding_real_analytic_manifolds_1958}, a Stein manifold $W$ admits an exhausting plurisubharmonic function $\rho: W\rightarrow \mathbb{R}$. For any regular value $c$ of $\rho$, the complex tangencies define a contact structure on the level set $M_{c}:=\{x\in W|\rho(x)=c\}$. We call the manifold $W_{c}:=\{x\in W|\rho(x)\leq c\}$ a \textit{Stein filling} of the contact manifold $M_{c}$.

Given a contact 3-manifold $(M,\xi)$, there are two natural questions one can ask: Is $(M,\xi)$  Stein fillable? and if yes, then is it possible to classify all the Stein fillings of $(M,\xi)$?

The first question has been addressed in terms of open book decompositions. For example, in \cite{Giroux_Correspondence}, Giroux gave a complete characterization of Stein fillability in terms of open book decomposition, see also \cite{Akbulut_Ozbagci_Stein_Surface_Lefschetz_fibrations, Loi_Piergallini_Lefschetz_fibrations}.

The second question has been answered for some specific contact 3-manifolds. Here we recall a few results in this direction. In \cite{Eliashberg_Filling_by_holomorphic_disks_and_applications}, Eliashberg showed that there is a unique Stein filling, up to symplectic deformation, of $S^3$ with the standard tight contact structure. In \cite{Stipsicz_gauge_theory_and_stein_fillings_of_certain_3_manifolds}, Stipsicz showed that there is a unique Stein filling, up to homeomorphism, of the Poincar\'{e} homology sphere $\Sigma(2,3,5)$ and the 3-torus $T^3$. In \cite{Wendl_Planar_open_books}, Wendl showed that there is a unique Stein filling, up to symplectomorphism, of the 3-torus $T^3$. In \cite{McDuff_Rational_ruled_surfaces}, McDuff showed that, for a universally tight contact structure, there is a unique Stein filling of a lens space $L(p,1)$ ($p\neq4$), and there are two Stein fillings of the lens space $L(4,1)$,  up to diffeomorphism.  In \cite{Lisca_Classification_of_Stein_fillings_on_lens_spaces}, Lisca classified the Stein fillings of universally tight lens spaces up to diffeomorphism.  In \cite{Plamenevskaya_VHMorris_Planar_open_books}, Plamenevskaya and Van Horn-Morris showed that there is a unique Stein filling, up to symplectic deformation, of $L(p,1)$ with a virtually overtwisted tight contact structure. In \cite{Kaloti_Stein_fillings_of_planar_open_books}, among other classes of manifolds, the first author classified Stein fillings of virtually overtwisted tight contact structures on the lens space $L(pm+1,m)$ for $p,m\geq1$, up to symplectic deformation. In \cite{Ohta_Ono_Classifications_of_Stein_fillings_1, Ohta_Ono_Classifications_of_Stein_fillings_2}, Ohta and Ono classified the Stein fillings of some links of simple singularities. In \cite{Schoenenberger_thesis_upenn}, Sch\"{o}nenberger showed that the Stein fillings of some contact Seifert fibered spaces are unique up to diffeomorphism. In \cite{Starkston_Symplectic_fillings_of_Seifert_fibered_spaces}, Starkston gave finiteness results and some classifications, up to diffeomorphism, of minimal strong symplectic fillings of certain contact Seifert fibered spaces over $S^{2}$.

In this paper, we classify the Stein fillings, up to diffeomorphism, or up to symplectic deformation, for an infinite family of contact 3-manifolds including an infinite number of contact hyperbolic 3-manifolds. This is the first known classification result for any contact hyperbolic $3$-manifolds.

These contact manifolds are supported by particular open books, which we describe now. Let $\Sigma$ be a compact planar surface with $n+p+q+1$ boundary components $c_0, c_1, \ldots, c_{n+p+q}$ as shown in Figure~\ref{fig:p14}, where $n,k,p,q \geq1$ and $n\geq k$. Let $\Phi$ be a diffeomorphism which is the composition of right handed Dehn twists written as $$\Phi=\tau_{1}^{m_1}\tau_{2}^{m_2}\ldots \tau_{n+q-1}^{m_{n+q-1}}\tau_{n+q+1}^{m_{n+q+1}}\ldots\tau_{n+p+q}^{m_{n+p+q}}\tau_{B_1}\tau_{B_2},$$ where $\tau_i$ is the positive Dehn twist about a simple closed curve parallel to the boundary component $c_i$, $m_i\geq 0$, and $\tau_{B_1}$, $\tau_{B_2}$ are positive Dehn twists along the simple closed curves $B_1$ and $B_2$ shown in Figure~\ref{fig:p14}.

\begin{figure}[htb]
\begin{overpic}
{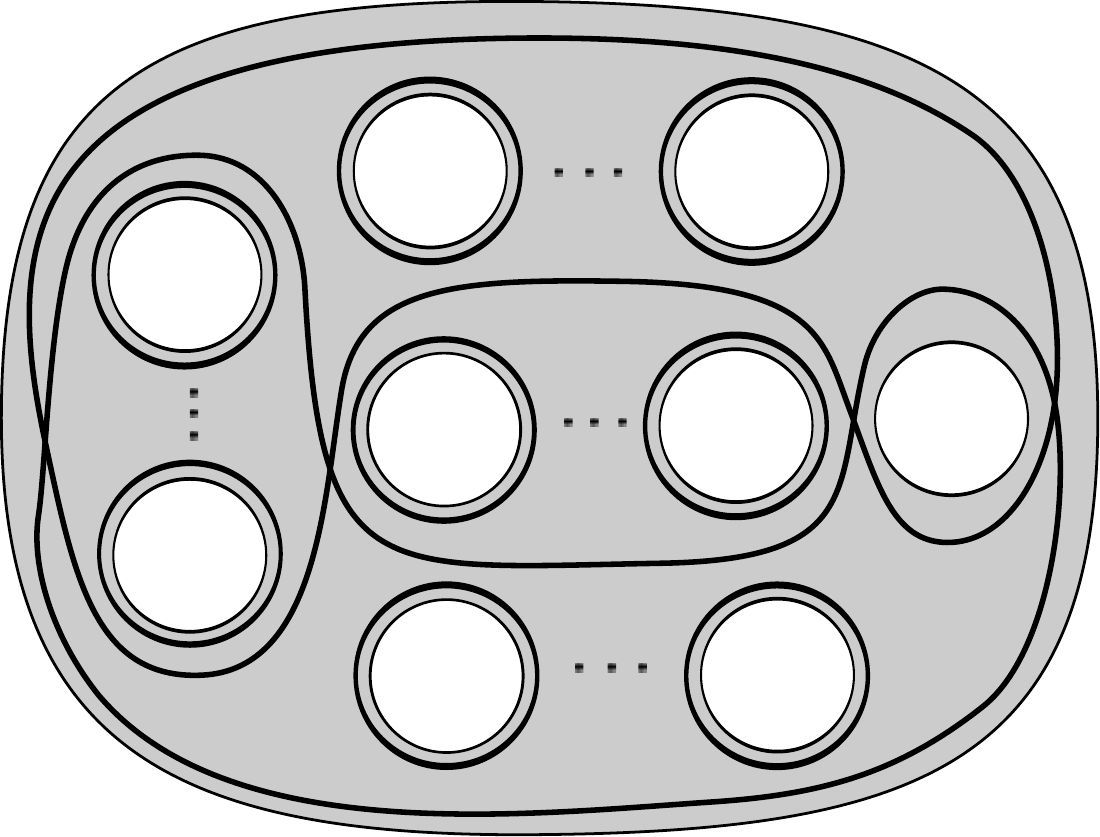}
\put(274, 11){$0$}
\put(222, 45){$1$}
\put(118, 45){$k-1$}
\put(53, 80){$k$}
\put(31.5, 160){$k+q-1$}
\put(113, 190){$k+q$}
\put(196, 190){$n+q-1$}
\put(264, 119){$n+q$}
\put(191.5, 116){$n+q+1$}
\put(106.25, 116){$n+q+p$}
\put(270, 190){$B_1$}
\put(270, 45){$B_2$}
\end{overpic}
\caption{A compact planar surface $\Sigma$ with $n+q+p+1$ boundary components.}
\label{fig:p14}
\end{figure}
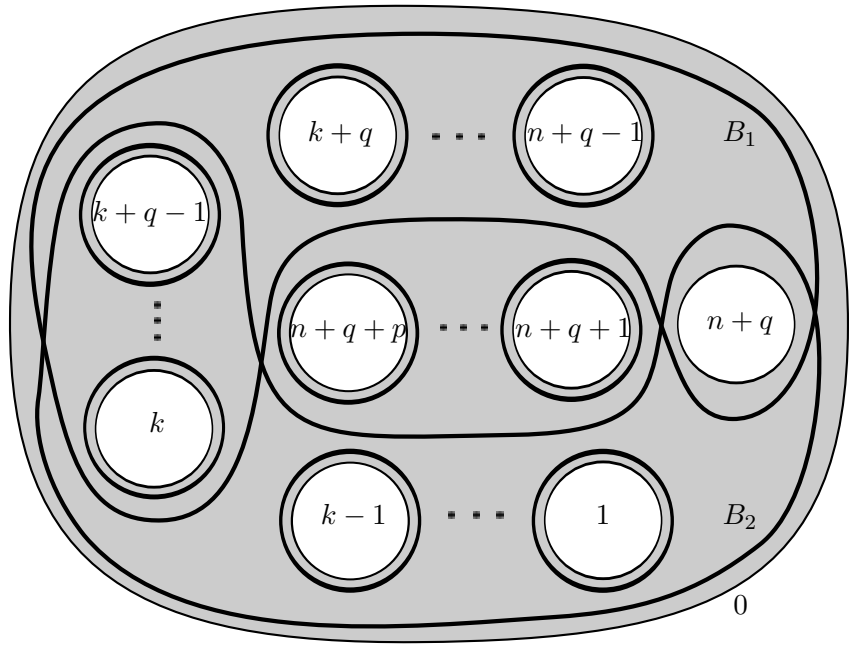

\begin{theorem}
\label{theorem_factorization_of_mapping_class}
Let $(M,\xi)$ be the contact 3-manifold supported by the open book $(\Sigma, \Phi)$. Then the contact $3$-manifold $(M,\xi)$ admits a unique Stein filling up to diffeomorphism.
\end{theorem}

The open book shown in Figure~\ref{fig:p14} covers a lot of interesting special cases. We describe a few of them below. In $(S^{3}, \xi_{std})$, let $L$ be a Legendrian twist knot, $K_{-2p}$, with Thurston-Bennequin invariant $-1$ and rotation number $0$, where $2p$ denotes the number of left-handed half twists. If $p=1$, then it is a right handed trefoil. Note that the maximal Thurston-Bennequin invariant is $1$ for the Legendrian twist knot $K_{-2p}$ at hand \cite{Etnyre_Ng_Vertesi_Twist_knots_classification}, and one stabilizes it twice to obtain the Legendrian knot $L$ and stabilize it further as in Figure~\ref{fig:Legendrian_surgery_diagram_for_infinite_family1}.
Let $n,k\geq 1$ be two integers such that $n \geq k$. Let $S^{n-k}_{+}S^{k-1}_{-}(L)$ be the result of $n-k$ positive stabilizations and $k-1$ negative stabilizations of $L$. Figure~\ref{fig:Legendrian_surgery_diagram_for_infinite_family1} depicts a Legendrian link in $(S^{3}, \xi_{std})$ one of whose components is $S^{n-k}_{+}S^{k-1}_{-}(L)$. The other components are all Legendrian unknots with Thurston-Bennequin invariant $-1$, pushed off $m_{i}$ times, where $m_i$ is a non-negative integer for $i=1,\ldots, k-1, k+1,\ldots, n$ if $k>1$ or $n>k$. Let $(M',\xi')$ denote the contact structure obtained by performing Legendrian surgery along the link given in Figure~\ref{fig:Legendrian_surgery_diagram_for_infinite_family1}.

\begin{figure}[htb]
\begin{overpic}
{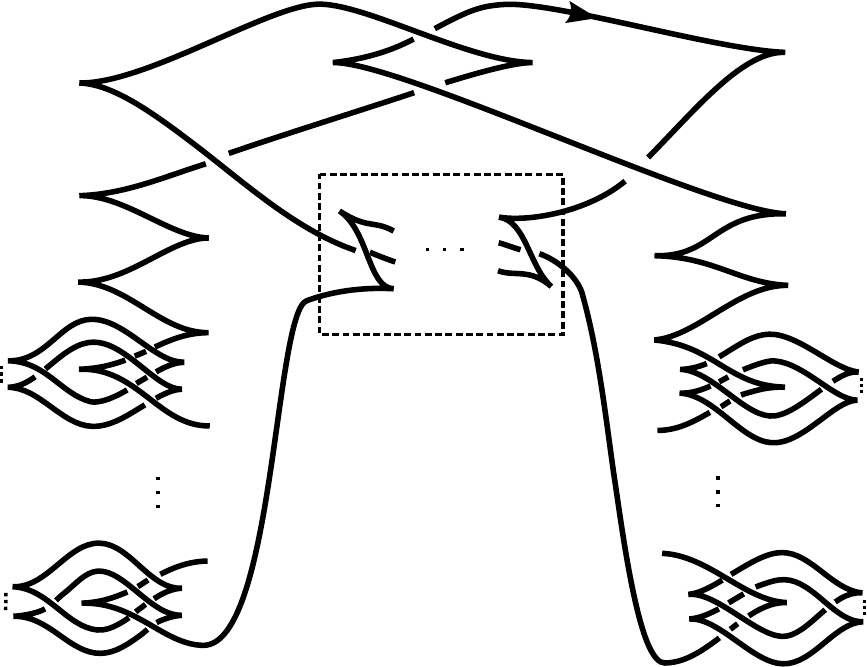}
\put(-45, 80){$m_{k+1}-1$}
\put(253, 80){$m_{k-1}-1$}
\put(-33, 16){$m_{n}-1$}
\put(253, 16){$m_{1}-1$}
\end{overpic}
\caption{A Legendrian link one of whose components is a Legendrian twist knot $K_{-2p}$, where the box consists of $2p-2$ Legendrian tangle $S$'s which is depicted below in Figure~\ref{fig:S}. There are $k-1$ upward cusps of the Legendrian $K_{-2p}$  each of which hooks $m_{i}-1$ Legendrian unknots for $i=1,\ldots, k-1$.  There are $n-k$ downward cusps of the Legendrian $K_{-2p}$  each of which hooks $m_{i}-1$ Legendrian unknots for $i=k+1,\ldots, n$.}
\label{fig:Legendrian_surgery_diagram_for_infinite_family1}
\end{figure}

\begin{figure}[htb]
\begin{overpic}
{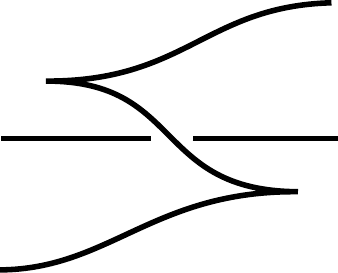}
\end{overpic}
\caption{A Legendrian tangle $S$.}
\label{fig:S}
\end{figure}

\begin{corollary}
\label{twist_knot_family}
The contact 3-manifold $(M',\xi')$ has a unique Stein filling up to diffeomorphism.
\end{corollary}

Another application of the above observation is classifying Stein fillings of manifolds obtained by Legendrian surgeries along some Legendrian $2$-bridge knots. Figure~\ref{fig:2-bridge_knots} depicts a 2-bridge knot $B(p,q)$, where $p,q$ are positive integers. If $q=1$, then it is the twist knot $K_{-2p}$.

\begin{figure}[htb]
\begin{overpic}
{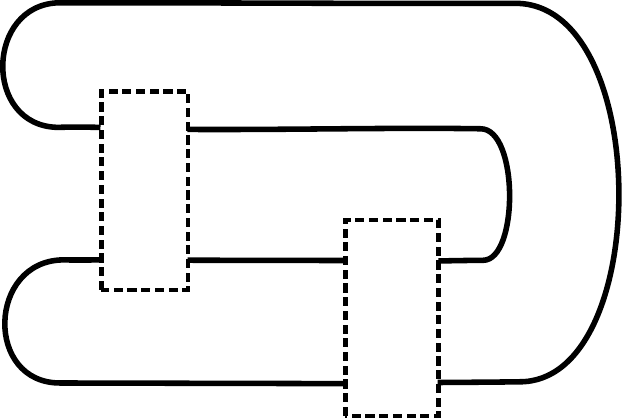}
\put(33, 64){$-2q$}
\put(105, 25){$-2p$}
\end{overpic}
\caption{A $2$-bridge knot $B(p,q)$ with $p,q > 0$. The boxes containing $-2p$ and $-2q$ denote $2p$ and $2q$ negative half twists, respectively. }
\label{fig:2-bridge_knots}
\end{figure}

\begin{corollary}
\label{2_bridge_family}
There is a Legendrian 2-bridge knot $B(p,q)$ with Thurston-Bennequin invariant $-1$ and rotation number $0$, such that the  Legendrian surgery on $(S^3, \xi_{std})$ along any of its stabilizations yields a contact 3-manifold with unique Stein filling up to diffeomorphism.
\end{corollary}

In addition to classifying Stein fillings up to diffeomorphism, we can classify Stein fillings of Legendrian surgeries along some Legendrian twist knots up to symplectic deformation. Even though these manifolds admit open books considered in Theorem~\ref{theorem_factorization_of_mapping_class}, we include a separate proof here because the notion of symplectic deformation is stronger than that of diffeomorphism.

\begin{theorem}
\label{trefoil_family}
If $L$ is  a Legendrian twist knot $K_{-2p}$ with Thurston-Bennequin invariant $-1$ and rotation number $0$, then the Legendrian surgery on $(S^3, \xi_{std})$ along any stabilization of $L$ yields a contact 3-manifold with unique Stein filling up to symplectic deformation equivalence.
\end{theorem}

If $p>1$, then the twist knot $K_{-2p}$ is hyperbolic. By the hyperbolic Dehn surgery theorem in \cite{Thurston_3_Manifolds_Kleinian_groups_and_hyperbolic_geometry}, Legendrian surgery on $(S^3, \xi_{std})$ along a Legendrian hyperbolic twist knot with sufficiently many stabilizations yields a contact hyperbolic 3-manifold. So, immediately, we have

\begin{corollary}
There are infinitely many contact hyperbolic 3-manifolds each of which admits a unique Stein filling up to symplectic deformation equivalence.
\end{corollary}

Rest of the paper is organized as follows. In Section~\ref{background} we recall some basic notions from mapping class groups and contact geometry. In Section~\ref{proofs_of_main_results} we give proofs of all our results.
\\

{\bf Acknowledgements}: Authors would like to thank John Etnyre and Dan Margalit for helpful conversations. In addition, authors would like to thank John Etnyre for thoroughly reading the first draft of the paper and giving invaluable comments.  This work was carried out while the second author was visiting Georgia Institute of Technology and he would like to thank them for their hospitality. The first author was partially supported by NSF grant DMS-0804820. The second author was partially supported by NSFC grant 11001171 and the China Scholarship Council grant 201208310626. Finally we would like to thank the referee for helpful comments and feedback.


\section{Background}
\label{background}

In this section, we recall some basic notions from mapping class groups and give an explicit presentation of the planar mapping class groups to be used throughout. We also describe a technique used to understand surgeries along Legendrian knots through open books.

\subsection{\textbf{Planar mapping class group}}

\label{planar_mapping_class_group}
Let $\mathbb{D}_n$ denote a sphere with $n+1$ open disks removed. We label the boundary components of $\mathbb{D}_n$ as $c_0, c_1,\dots,c_{n}$. By fixing an outer boundary component, denoted by $c_{0}$, we can embed $\mathbb{D}_n$ in $\mathbb{R}^2$. The mapping class group of $\mathbb{D}_n$ is a group of self diffeomorphisms of $\mathbb{D}_n$ up to isotopy such that each diffeomorphism fixes the boundary pointwise. We will denote the mapping class group by $Map(\mathbb{D}_n,\partial \mathbb{D}_n)$ and will call it planar mapping class group. It is a well known fact that the mapping class group of any surface is finitely presented, see \cite{Farb_Margalit_Primer}. We use a particular presentation of the mapping class group of a planar surface due to Margalit and McCommand, \cite{McCommand_Margalit_Braid_Group}.  We briefly describe their presentation.

For this presentation, we assume that the boundary components are arranged at vertices of a regular $n$-gon. We call a curve convex, if it is isotopic to the boundary of a convex hull of a collection of boundary components. A Dehn twist about a convex curve is called convex Dehn twist. According to \cite{McCommand_Margalit_Braid_Group}, mapping class group of $\mathbb{D}_n$ is generated by convex twists. The relations are given by

\begin{enumerate}

\item $\tau_{\mathcal{A}} \tau_{\mathcal{B}} = \tau_{\mathcal{B}} \tau_{\mathcal{A}}$ if and only if $\mathcal{A}$ is disjoint from $\mathcal{B}$. Here $\mathcal{A}$ and $\mathcal{B}$ are simple closed curves.

\item $\tau_{\mathcal{A}} \tau_{\mathcal{B}} \tau_{\mathcal{C}} \tau_{\mathcal{A \cup B \cup C}} = \tau_{\mathcal{A \cup B}} \tau_{\mathcal{B \cup C}} \tau_{\mathcal{A \cup C}}$, where $\mathcal{A}, \mathcal{B}, \mathcal{C}$ are disjoint collections of boundary components and Dehn twists are convex Dehn twist about them. In addition, we require that the boundary components are ordered such that the cyclic clockwise ordering of boundary components in $\mathcal{A}$  followed by those in $\mathcal{B}$ followed by those in $\mathcal{C}$ induces the cyclic clockwise ordering of boundary components in $\mathcal{A \cup B \cup C}$.

\end{enumerate}

Now we define homomorphisms from $Map(\mathbb{D}_n, \partial \mathbb{D}_n)$ to $\mathbb{Z}$, which define multiplicities associated to Dehn twists. Similar homomorphisms were defined in~\cite{Plamenevskaya_VHMorris_Planar_open_books}. Recall that $Map(\mathbb{D}_2, \partial \mathbb{D}_2)$ is isomorphic to $\mathbb{Z}^3$ which is generated by Dehn twists about each of the boundary components.  Let  $\Phi \in Map(\mathbb{D}_n, \partial \mathbb{D}_n)$ be a word written as product of positive Dehn twists. Let $c_i$ and $c_j$ be any boundary components in $\mathbb{D}_n$ other than $c_0$.

\begin{definition}[\textbf{Joint Multiplicity}] \label{joint_multiplicity}
Capping off all the boundary components of $\mathbb{D}_n$ except $c_i, c_j, c_{0}$ with disks, we obtain a map to $ \mathbb{Z}$ which just counts the number of Dehn twists about the curve parallel to the outer boundary $c_0$. We call this the \textit{joint multiplicity} of boundary  components $i$ and $j$ and denote it by $M_{i,j}(\Phi)$.

\end{definition}

\begin{definition}[\textbf{Mutiplicity}]\label{multiplicity}
Cap off all the holes except the boundary components $c_i$ and $c_{0}$. This induces a map from $Map(\mathbb{D}_n, \partial \mathbb{D}_n)$  to $Map(\mathbb{D}_1, \partial \mathbb{D}_1) \cong \mathbb{Z}$ and the map counts the Dehn twists about the boundary parallel curve. We call this the \textit{multiplicity} of the boundary component $c_i$. Denote it by $M_i(\Phi)$.
\end{definition}

\subsection{\textbf{Open book decompositions, Lefschetz fibrations and Stein fillings}}

We recall briefly notion of open book decompositions.  Given a closed oriented $3$-manifold $M$, an \textit{embedded open book decomposition} of $M$ is a pair $(B,\pi)$ where
\begin{itemize}
\item $B$ is an oriented  link in $M$ called the \textit{binding} of the open book decomposition and,

\item $\pi : M\setminus B \rightarrow S^1$ is a fibration such that $\pi^{-1}(\theta)$ is the interior of a compact surface $S_{\theta}$ with $\partial S_{\theta}=B$ for all $\theta\in S^1$.
\end{itemize}

Given an embedded open book decomposition as above,  since $\pi$ is a locally trivial fibration over $S^1$, it is completely specified by a diffeomorphism of the fiber surface $S$. To see this, think of $S^1$ as the interval $[0,2\pi]$, with end points identified. Since $[0,2\pi]$ is contractible any fiber bundle over it is trivial. Hence, the original fibration $ \pi: M\setminus B \rightarrow S^1$ can now be obtained by gluing the surfaces $S \times \lbrace 0 \rbrace$ and $S \times \lbrace 2 \pi \rbrace$. Hence we get an alternate description of an open book decomposition for $M$ called an \textit{abstract open book decomposition} which is defined as follows.
\begin{itemize}
\item $S$ is an oriented compact surface with boundary called the \textit{page} of	 the open book decomposition.

\item $\phi: S \rightarrow S$ is a diffeomorphism of $S$ such that $\phi|_{\partial S}$ is the identity. The diffeomorpshim $\phi$ is called the \textit{monodromy} of the open book decomposition.
\end{itemize}

Given an abstract open book decomposition $(S,\phi)$, we get a $3$-manifold $M_{\phi}$ as $S \times [0,1]/ \sim$, where $\sim$ is the equivalence relation $(x,1)\sim(\phi(x),0)  $ for $x \in S$, and $(y,t)\sim (y, t')$ for $y\in\partial S$, $t,t'\in [0,1]$.

Note that an abstract open book decomposition $(S,\phi)$ determines an oriented closed manifold $M_{\phi}$ and an embedded  open book decomposition $(B_{\phi}, \pi_{\phi})$ up to diffeomorphism.

For an embedded open book decomposition $(B,\pi)$, a \textit{positive stabilization} is an operation of plumbing a positive Hopf band to $(B,\pi)$.  For an abstract open book decomposition, a \textit{positive stabilization} changes the page $S$ by attaching a $1$-handle to get a new surface $S'$. The monodromy of the open book decomposition changes from $\phi$ to $\phi \cdot \tau_c$, where $c$ is curve which intersects the co-core of the attached $1$-handle exactly once and $\phi$ is extended to the new surface by identity. For an abstract open book decomposition $(S,\phi)$, a \textit{conjugation} is an operation to replace the monodromy $\phi$ by $f\phi f^{-1}$, where $f$ is a self-diffeomorphism of $S$ which is identity on the boundary.

A contact structure $\xi$ on $M$ is \textit{supported} by  an embedded  open book decomposition $(B,\pi)$  of $M$ if $\xi$ can be isotoped through contact structures so that there is a $1$-form $\alpha$ for $\xi$ such that:
\begin{itemize}
\item $\alpha > 0$ on $B$.
\item $d\alpha$ is a positive area form on each page $S_{\theta}$ of the open book decomposition.
\end{itemize}
A contact 3-manifold is \textit{supported} by  an abstract  open book decomposition $(S,\phi)$ of  if $\xi$ is supported by an embedded open book decomposition determined by $(S,\phi)$.

Given any abstract, or embedded open book decomposition $(B,\pi)$ of $M$, according to a construction of Thurston and Winkelnkemper \cite{Thurston_Wilkenkeper_Contact_structures_from_open_books}, it supports a contact structure $\xi$. Giroux proved that given a contact structure $\xi$ on $M$, one can find an embedded open book decomposition of $M$ supporting $\xi$.

\begin{theorem}[Giroux 2000, \cite{Giroux_Correspondence}]
\label{Giroux 1-1}
Let $M$ be a closed oriented 3-manifold. Then there is a one to one correspondence between the oriented contact structures on $M$ up to isotopy and the open book decompositions of $M$ up to positive stabilization, and there is a one to one correspondence between the oriented contact structures on $M$ up to isomorphism and the abstract open book decompositions of $M$ up to positive stabilization and conjugation.
\end{theorem}

From now on, unless explicitly specified, an open book decomposition will mean the abstract one.

Following Theorem~\ref{Giroux 1-1}, Giroux \cite{Giroux_Correspondence}, Loi-Piergallini \cite{Loi_Piergallini_Lefschetz_fibrations}, Akbulut-Ozbagci \cite{Akbulut_Ozbagci_Stein_Surface_Lefschetz_fibrations} showed that a contact manifold $(M,\xi)$ is Stein fillable if and only if there exists an open book decomposition $(S, \phi)$ supporting $(M,\xi)$ such that the monodromy is written as a product of positive Dehn twists. The proof of this result tells us how to construct a Stein filling from $(S, \phi)$ and a given a positive factorization of $\phi$. We briefly recall it here. To do so we first need to recall a few facts about Lefschetz fibrations. We refer the reader to \cite{Gompf_Stipsicz} for more details.

Suppose $X$ is an oriented 4-manifold, and $D^2$ is an oriented 2-dimensional disk. The smooth map $f:X\rightarrow D^2$ is a \textit{Lefschetz fibration} if $df$ is onto with finitely many exceptional points in the interior of $D^2$, the map $f$ is locally trivial in the complement of these exceptional points, and around each of the exceptional points, $f$ can be modelled in some choice of
complex coordinates by $f(z_{1}, z_{2})=z_{1}^{2}+z_{2}^{2}$. If $x \in D^2$ is a regular value, then $f^{-1}(x)$ is a compact smooth surface $S$. This surface $S$ is called the \textit{regular fiber}. At each critical value $p \in D^2$, $f^{-1}(p)$ is a singular surface with exactly one positive transverse self-intersection point, we call it a \textit{singular fiber}. One can assume that all the critical values, $p_1,\dots, p_n$ are isolated and monodromy of the fibration around each critical value $p_i$, is specified by an element of $Map(S,\partial S)$. This element is given by a right handed Dehn twist about a non trivial curve $\alpha_i \subset S$. The \textit{global monodromy} of the fibration is given by the product of Dehn twists $\tau_{\alpha_1}, \tau_{\alpha_2}, \dots, \tau_{\alpha_n}$.

For positive factorizations, a \textit{Hurwitz move} is defined by  $$\tau_{\alpha_1} \dots \tau_{\alpha_i} \tau_{\alpha_{i+1}} \dots \tau_{\alpha_n} \sim  \tau_{\alpha_1} \dots  \tau_{\alpha_{i+1}} \tau_{\tau^{-1}_{\alpha_{i+1}}(\alpha_i)} \tau_{\alpha_{i+2}} \dots \tau_{\alpha_n}, $$ for  $1 \leq i < n$. Two positive factorizations are \textit{Hurwitz equivalent} if one of them can be obtained from the other by finitely many Hurwitz moves.

In addition, to remove the dependence on choice of the reference fiber, the Dehn twists are thought of as elements of the mapping class of an abstract surface $S$. This requires the choice of an identification diffeomorphism and induces an equivalence relation on the set of mapping class group factorizations: \textit{ global conjugation}. It is defined as $$\tau_{\alpha_1} \tau_{\alpha_2}  \dots \tau_{\alpha_n} \sim  \tau_{f(\alpha_1)} \tau_{f(\alpha_2)} \dots  \tau_{f(\alpha_n)}, $$ where $f$ is a diffeomorphism of $\Sigma$ which is an element of $Map(S,\partial S)$.

Given an abstract open book decomposition $(S,\phi)$ of the manifold $(M,\xi)$ and a positive factorization of $\phi$ by $\tau_{\alpha_1}  \dots \tau_{\alpha_n}$, where $\alpha_1,\dots,\alpha_n$ are nontrivial curves in $S$, one can construct a Lefschetz fibration $X$ which has global monodromy $\tau_{\alpha_1} \dots \tau_{\alpha_n}$, and is diffeomorphic to $(S\times D^2)\bigcup(\cup_{i=1}^{n} H_{i})$, where $H_i$ is a 4-dimensional 2-handle attached along $\alpha_{i}$ in a fiber of $S \times\partial D^{2} \rightarrow \partial D^2$ with relative framing $-1$. Note that the diffeomorphism type of $X$ is unique up to Hurwitz equivalence and global conjugation of the positive factorization of $\phi$. According to \cite{Eliashberg_Stein_handle_attachment}, $X$ admits a Stein structure, and $\partial X$  is diffeomorphic to $M$. Furthermore, the contact structure on $M$ induced by this Stein filling agrees with the contact structure supported by the open book decomposition $(S,\phi)$ through the Giroux correspondence. Hence, $ X $ gives a Stein filling of $(M,\xi)$.

Conversely, given a Stein filling, $X$, of $(M,\xi)$, one can construct a Lefschetz fibration of $X$ (see \cite{Akbulut_Ozbagci_Stein_Surface_Lefschetz_fibrations, Loi_Piergallini_Lefschetz_fibrations}) such that $\partial X = M$ has a natural open book decomposition with monodromy written as a positive factorization and supports $\xi$.
Hence to get an upper bound on the number of Stein fillings of $(M, \xi)$, one will have to find all compatible open books and then find all possible ways of factorizing a given monodromy in terms of positive Dehn twists. But in the case of manifolds supported by planar open book decompositions this problem is approachable due to the following theorem of Wendl.

\begin{theorem}[Wendl 2010, \cite{Wendl_Planar_open_books}]
\label{Wendl_Planar_open_books}
Suppose $(M,\xi)$ is supported by a planar open book decomposition. Then every strong symplectic filling $(X,\omega)$ of $(M,\xi)$ is symplectic deformation equivalent to a blow-up of an allowable Lefschetz fibration compatible with the given open book decomposition of $(M,\xi)$.
\end{theorem}

So in case of contact manifolds supported by planar open book decompositions, this theorem tells us that we  need to find all possible ways of factorizing the given monodromy in terms of positive Dehn twists, to get an upper bound on the number of Stein fillings. Then one can try to classify the Stein fillings by realizing this upper bound. This  approach to classifying Stein fillings has proved to be successful in \cite{Plamenevskaya_VHMorris_Planar_open_books,Kaloti_Stein_fillings_of_planar_open_books}.

\subsection{\textbf{Open book decompositions for manifolds obtained by surgery along knots.}}

It follows from Giroux correspondence that any Legendrian knot $L$ in $(M,\xi)$ can be embedded in a page of an open book supporting $\xi$ as an essential simple closed curve, see \cite{Etnyre_Lectures_on_open_book_decompositions}. One can prove that if $L$ is a Legendrian knot on a page of open book $(S,\phi)$ supporting $(M,\xi)$ and $(M_{(-1)}(L), \xi_L)$ is the contact structure obtained from $(M, \xi)$ by Legendrian surgery along $L$, then $(M_{(-1)}(L), \xi_L)$ is supported by the open book decomposition $(S,\phi \cdot \tau_L)$. Here $L$ is thought of as a simple closed curve embedded in the surface $S$.

Given a Legendrian knot $L$, there is a natural operation called positive/negative stabilization of $L$ that can be used to get another Legendrian knot in the same knot type. For a Legendrian knot in $\mathbb{R}^3$ with its standard contact structure, positive (resp. negative) stabilization is achieved by adding a zigzag to the front projection of the Legendrian knot such that rotation number of the Legendrian knot increases (resp. decreases)  by $1$. Stabilization is a well-defined operation, that is, it does not depend on the point at which zigzags are added.   See \cite{Etnyre_Lectures_on_open_book_decompositions} for details and proof of the following lemma.

\begin{lemma}
\label{lemma:knot_stabilization}
Let $(S, \phi)$ be an open book decomposition supporting the contact structure $\xi$ on $M$. Suppose $L$ is a Legendrian knot in $M$ that lies in the page $S$. If we stabilize $(S,\phi)$ as shown in Figure~\ref{Knot_stabilization}, then we may isotope the page of the open book so that positive (negative) stabilization of $L$ appears on the page $S$ as shown in Figure~\ref{Knot_stabilization}.
\end{lemma}

\begin{figure}[htb]
\begin{overpic}
{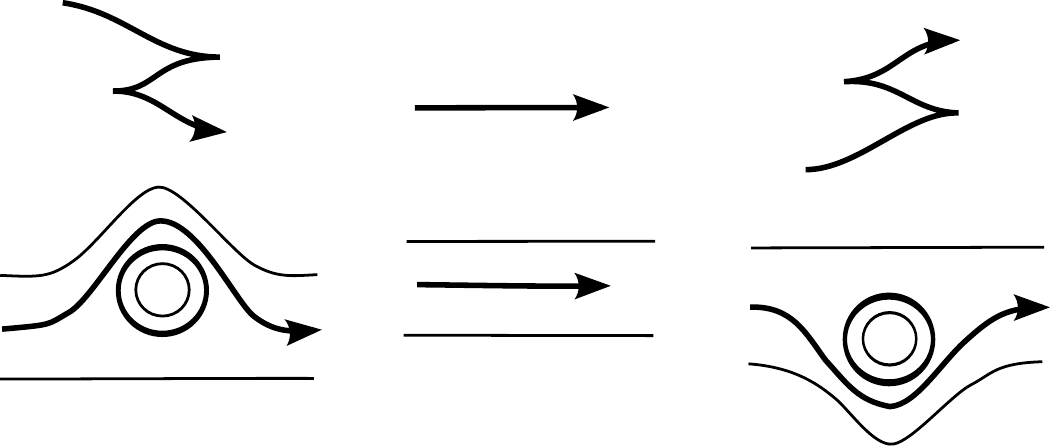}
\put(55, 30){$+$}
\put(270, 40){$+$}
\end{overpic}
\caption{Stabilizing the knot on the page of an open book. Middle figure shows a Legendrian knot on a page of open book. Left figure shows how to stabilize the knot positively while the right figure shows how to stabilize the knot negatively.}
\label{Knot_stabilization}
\end{figure}

\section{\textbf{Classification of Stein fillings}}

\label{proofs_of_main_results}

We begin by observing a purely combinatorial lemma. The purpose of this lemma is to get restrictions on the curves which can appear in any positive factorization of the given monodromy in terms of Dehn twists. Refer to Figure~\ref{fig:p14} for the notation used below.

\begin{lemma}
\label{lemma_combinatorial}
Any positive factorization of $\Phi$ must be given by the product of Dehn twists  $\tau_{1}^{m_1},$ $\tau_{2}^{m_2},$ $\ldots,$ $\tau_{n+q-1}^{m_{n+q-1}}, $ $\tau_{n+q+1}^{m_{n+q+1}},$ $\ldots, $ $\tau_{n+p+q}^{m_{n+p+q}}$, and the Dehn twists $\tau_{B'_1}$ and $\tau_{B'_2}$ where $B'_1$ encloses the same holes as  $B_1$, and $B'_2$ the same holes as $B_2$.
\end{lemma}

\begin{proof} Recall from Section~\ref{background} that, $M_{i,j}$ denotes the joint multiplicity of the mapping class $\Phi$ about the $i^{th}$ and $j^{th}$ boundary components, and $M_{i}$ denotes the multiplicity of the mapping class $\Phi$ about the $i^{th}$ boundary component.

Since $M_{n+q}=2$ and $M_{i, n+q}=2$ for $i\in\{k, k+1, \ldots, k+q-1\}$, there are exactly two monodromy curves, say $B'_1$ and $B'_2$, enclosing $c_{n+q}$ and $c_{i}$  for $i\in\{k, k+1, \ldots, k+q-1\}$. Since $M_{n+q, j}=0$ for $j\in\{n+q+1, n+q+2, \ldots, n+q+p\}$,  $M_{r, n+q}=1$ for $r\in\{1, \ldots, k-1, k+q,\ldots, n+q-1\}$, and $M_{s, t}=0$ for $s\in\{1, \ldots, k-1\}$ and $t\in\{k+q,\ldots, n+q-1\}$, the two monodromy curves $B'_1$ and $B'_2$ enclose $\{c_k, \ldots,  c_{k+q-1}, c_{k+q}, \ldots, c_{n+q-1}, c_{n+q}\}$ and  $\{c_1,  \ldots, c_{k-1}, c_{k}, \ldots, c_{k+q-1}, c_{n+q}\}$  respectively.

For $j\in\{n+q+1, n+q+2, \ldots, n+q+p\}$,  $M_{j}=m_{j}$ and $M_{i, j}=0$ for any $i\in\{1,2,\ldots, n+q+p\}$ and $i\neq j$. So $c_{j}$ is enclosed solely by $m_j$ boundary parallel monodromy curves.

For $i\in\{1,\ldots, n+q-1\}$, there are no non-boundary-parallel monodromy curves, other than $B'_1$ and $B'_2$, enclosing $c_i$.  Suppose otherwise, then for some  $j, h\in\{1,\ldots, n+q-1\}$,  there is a monodromy curve, other than $B'_1$ and $B'_2$,  enclosing $c_j$ and $c_h$.
If either $j$ or $h$ does not belong to $\{k, k+1, \ldots, k+q-1\}$, then  $j$ and $h$ cannot belong to $\{1, \ldots, k-1\}$ and $\{k+q,\ldots, n+q-1\}$, respectively. So $M_{j,h}\geq 2$. However, from the original positive decomposition of $\Phi$, we have $M_{j,h}=1$. So we arrive at a contradiction. If both $j$ and $h$ belong to $\{k, k+1, \ldots, k+q-1\}$, then $M_{j,h}\geq 3$. However, also from the original positive decomposition of $\Phi$, we have $M_{j,h}=2$. So we arrive at a contradiction as well.

Hence for $i\in\{1,\ldots, n+q-1\}$, there are $m_{i}$ boundary parallel monodromy curves enclosing $c_i$.
\end{proof}

\begin{remark}
With the notation as in above lemma, any other factorization of $\Phi$ can be written as $\tau_{1}^{m_1} \tau_{2}^{m_2} \ldots \tau_{n+q-1}^{m_{n+q-1}} \tau_{n+q+1}^{m_{n+q+1}} \ldots \tau_{n+p+q}^{m_{n+p+q}} \tau_{B'_1} \tau_{B'_2}$ up to Hurwitz equivalence. To see this, recall that since boundary Dehn twists commute with every diffeomorphism we can move them all to the left in the factorization as written above. Now product of Dehn twists $\tau_{B'_1}$ and $ \tau_{B'_2}$ is on the right side of this positive factorization. Hurwitz move on the product of Dehn twists can potentially change the homotopy class of both the curves $B'_1$ and $B'_2$, to say $B''_1$ and $B''_2$. But still $B''_i$ enclose the same set of holes as $B'_i$ for $i =1,2$. With abuse of notation we still call these new set of curves as $B'_1$ and $B'_2$, as Lemma~\ref{lemma_combinatorial} only specifies the curves $B'_1$ and $B'_2$ up to the set of holes enclosed by each of these curves. So using commutativity of boundary parallel Dehn twists and Hurwitz moves one can arrange the factorization as above.

In our case, since we only have two non boundary parallel monodromy curves, a Hurwitz move is also a global conjugation.
\end{remark}

\subsection{Positive factorizations.}

In this subsection, we prove that $\tau_{B_1}\tau_{B_2}$ has at most $2$ different positive factorizations, up to a global conjugation, in $Map(\Sigma, \partial \Sigma)$ for some simple cases of the surface $\Sigma$. We will reduce the above factorization problem to these simple cases later.

In proving these results, first step will be to get restrictions on intersection number of curves $B'_1$ and $B'_2$. To make sense of intersection numbers of curves, we assume for the rest of the paper that any two curves are isotoped to intersect minimally.

\begin{lemma}
\label{lemma_uniqueness_of_positive_factorization}
Let $\Sigma$ be the surface in Figure~\ref{fig:p14} with  $k=1$ and  $n=1$. Suppose $p=q=1$.  Then $\tau_{B_1} \tau_{B_2}\in Map(\mathbb{D}_3, \partial \mathbb{D}_3)$  has at most two positive factorizations up to a global conjugation.
\end{lemma}

\begin{proof}
Since $k=n=1$, there is no hole which is enclosed by $B_2$ but not by $B_1$, and there is no hole which is enclosed by $B_1$ but not by $B_2$.

By Lemma~\ref{lemma_combinatorial}, any positive factorization of $\tau_{B_1} \tau_{B_2}$ is $\tau_{B'_1} \tau_{B'_2}$ up to a global conjugation, where $B'_1$ and $B'_2$ enclose the same set of holes as $B_1$ and $B_2$, respectively.

We would like to get more information about the curves $B'_1$ and $B'_2$. To do that, we use a theorem by Thurston \cite[Theorem 7]{thurston_geometry_and_dynamics_of_diffeomorphism_of_surface}. This was suggested by Dan Margalit to us.

We assume that each of the boundary components is filled by a disk with one puncture. To avoid confusion, we will call this surface $\mathbb{D}_3$ still. Note that the curves $B_1$ and $B_2$ fill the surface $\mathbb{D}_3$. As explained in~\cite[Expos\'{e}~$13$]{FLP_Thurston's_work_on_surfaces}, one can construct a singular flat Euclidean structure and a representation of the subgroup of $Map(\mathbb{D}_3, \partial \mathbb{D}_3)$ generated by $\tau_{B_1}$ and $\tau_{B_2} $. Precisely (see \cite[Theorem 7]{thurston_geometry_and_dynamics_of_diffeomorphism_of_surface}, and  \cite[Theorem 14.1]{Farb_Margalit_Primer}),  suppose $I(B_1,B_2)$ denote the geometric intersection number of $B_1$ and $B_2$, then there is a representation
$\rho: \langle \tau_{B_1}, \tau_{B_2} \rangle \rightarrow PSL(2,\mathbb{R})$ given by
\begin{center}
$\tau_{B_1}  \mapsto \begin{bmatrix} 1 & I(B_1,B_2) \\ 0 & 1  \end{bmatrix}, \tau_{B_2} \mapsto \begin{bmatrix} 1 & 0 \\ -I(B_1,B_2) & 1 \end{bmatrix}$
\end{center}
with the following properties:

\begin{enumerate}

\item An element $g \in \langle \tau_{B_1}, \tau_{B_2} \rangle$ is periodic, reducible, or pseudo-Anosov according to whether $\rho(g)$ is elliptic, parabolic, or hyperbolic.

\item When $\rho(g)$ is hyperbolic the pseudo-Anosov mapping class $g$ has stretch factor equal to the larger of the absolute values of the two eigenvalues of $\rho(g)$.

\end{enumerate}

The matrix $\rho(g)$ is called an affine representative for $g$. So  $\tau_{B_1}$ has affine representative given by $\begin{bmatrix}
1 & 4 \\
0 & 1
\end{bmatrix} $,
and $\tau_{B_2}$ has an affine representative $
\begin{bmatrix}
1 & 0 \\
-4 & 1
\end{bmatrix}
$.
Thus, we obtain an affine representative for $\tau_{B_1} \tau_{B_2}$. It is $
\begin{bmatrix}
-15 & 4 \\
-4 & 1
\end{bmatrix}$.
This matrix has trace $-14$, so $\tau_{B_1} \tau_{B_2}$ has a pseudo-Anosov representative with stretch factor the larger of the absolute values of the two eigenvalues, that is $7+4\sqrt{3}$.  Note that the stretch factor of a pseudo-Anosov representative of a pseudo-Anosov diffeomorphism is in fact an invariant of the pseudo-Anosov diffeomorphism. This is because two homotopic pseudo-Anosov representatives are conjugate by a diffeomorphism isotopic to the identity (\cite[Th\'{e}or\`{e}me 12.5]{FLP_Thurston's_work_on_surfaces}), and any two conjugate pseudo-Anosov representatives have the same stretch factors (\cite[Page 406]{Farb_Margalit_Primer}).

Since $\tau_{B'_1} \tau_{B'_2}$, as a conjugation of $\tau_{B_1} \tau_{B_2}$, is also pseudo-Anosov, $B'_1$ and $B'_2$ have to intersect and fill the surface $\mathbb{D}_3$.  Otherwise there is a non-boundary-parallel simple closed curve which is invariant by $\tau_{B'_1} \tau_{B'_2}$. This is impossible for a pseudo-Anosov diffeomorphism.

Assume that $I(B'_1,B'_2) = z$, where $z$ is a non-negative integer. As above we obtain the affine representative for
$\tau_{B'_1} \tau_{B'_2}$. It is
$\begin{bmatrix}
1 & z \\
0 & 1
\end{bmatrix}
\begin{bmatrix}
1 & 0 \\
-z & 1
\end{bmatrix}=
\begin{bmatrix}
1-z^2 & z \\
-z & 1
\end{bmatrix} $.
Since $\tau_{B'_1} \tau_{B'_2}$ and $\tau_{B_1} \tau_{B_2}$ are conjugate, they have the same stretch factors (\cite[Page 406]{Farb_Margalit_Primer}). So $\frac{1}{2}(z^{2}-2+z\sqrt {z^{2}-4})=7+4\sqrt{3}$, and $z=4$.

Now we conjugate $\tau_{B'_1} \tau_{B'_2}$ by a diffeomorphism which takes  $B'_1$ to $B_1$, and $B'_2$ to a curve $B''_2$.  We know that $B''_2$ and $B_1$ intersect in exactly $4$ points.  Since $B''_2$ and $B_2$ represent the same homology classes in $H_1(\mathbb{D}_3)$, we know that the algebraic intersection number of $B''_2$ and $B_2$ with each nontrivial arc in the surface is the same. In particular, for the arc $\gamma$ shown in the left of Figure~\ref{fig:open_book_and_arc_a_shown}, the algebraic intersection number of $B''_2$ and $\gamma$ is $0$, and the  geometric intersection number of $B''_2$ and $\gamma$ is even.

 If $I(B''_2,\gamma) = 0$, then it is easy to see that $B''_2$ is isotopic to $B_1$. This is impossible by the above argument.

\begin{figure}[htb]
\begin{overpic}
{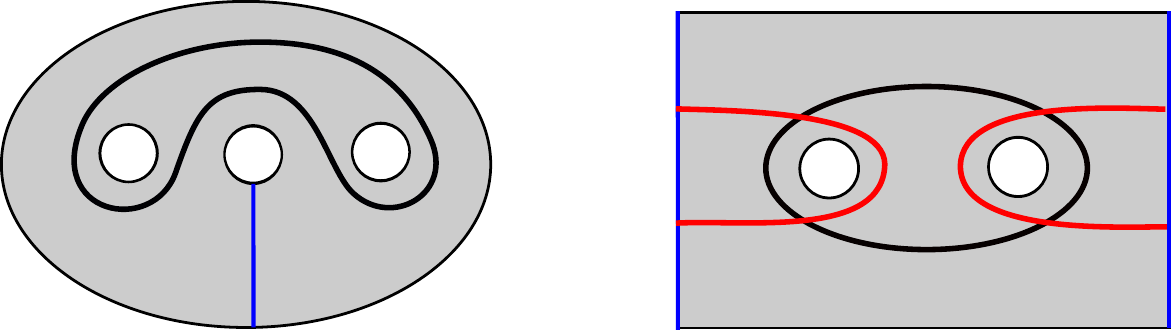}
\put(130, 10){$0$}
\put(34, 50){$1$}
\put(108, 50){$2$}
\put(70, 50){$3$}
\put(75, 20){$\gamma$}
\put(300, 83){$0$}
\put(300, 3){$3$}
\put(234, 45){$1$}
\put(290, 45){$2$}
\end{overpic}
\caption{Arc $\gamma$ along which the surface is cut open is shown in the left picture. The right picture shows the cut open surface.}
\label{fig:open_book_and_arc_a_shown}
\end{figure}
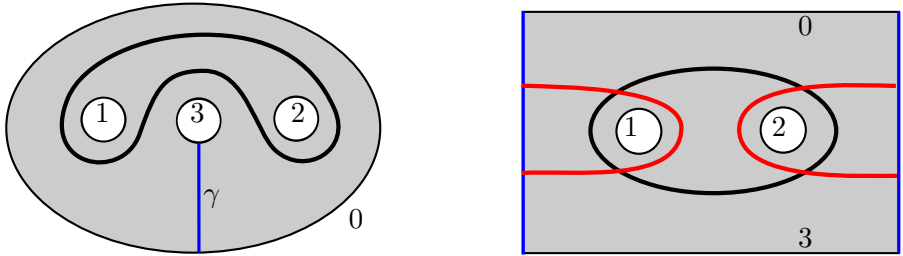

Suppose $I(B''_2,\gamma)\geq 2$. We cut the surface $\mathbb{D}_3$ open along the arc $\gamma$, and think of the resulted surface as a pair of pants with the outer boundary drawn as a rectangle.  See the right of Figure~\ref{fig:open_book_and_arc_a_shown}. Under this operation, $B''_2$ is cut into a collection of properly embedded arcs which are pairwise disjoint.  Each of these  arcs is one of the following three types.
\begin{itemize}
\item \textbf{Type~I:} Both  end points of the arc are on the left edge of  the rectangle.

\item \textbf{Type~II:} Both  end points of the arc are on the right edge of  the rectangle.

\item \textbf{Type~III:} The arcs have one end point on the left edge and the other on the right edge of  the rectangle.
\end{itemize}

Since $B_1$ is parallel to the outer boundary of the rectangle, each of Type I, II and III  arcs intersects the curve $B_1$ in 2 points or is disjoint with $B_1$. It is easy to see that each of Type I, II, and III arcs intersects the curve $B_1$ in exactly 2 points.

For $B''_2$ to be a simple closed curve enclosing holes $c_1$ and $c_2$, there is at least one arc of Type I and at least one arc of Type II. Since  $I(B''_2,B_1) = 4$, $B''_2$ is cut open into a Type I arc and a Type II arc.

If the Type I arc encloses the hole $c_1$ with the left edge of the rectangle, and the Type II arc encloses the hole $c_2$ with the right edge of the rectangle, then some power of  $\tau_{B_1}$ will send $B''_2$ to be the one which is formed by the two arcs shown in the right of Figure~\ref{fig:open_book_and_arc_a_shown}. Hence  $\tau_{B_1} \tau_{B''_2}$ is conjugate to $\tau_{B_1} \tau_{B_2}$.

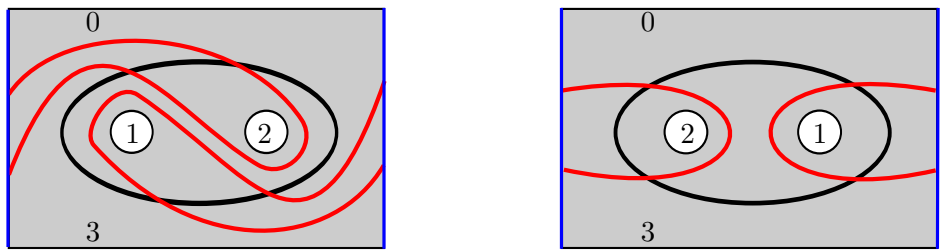
\begin{figure}[htb]
\begin{overpic}
{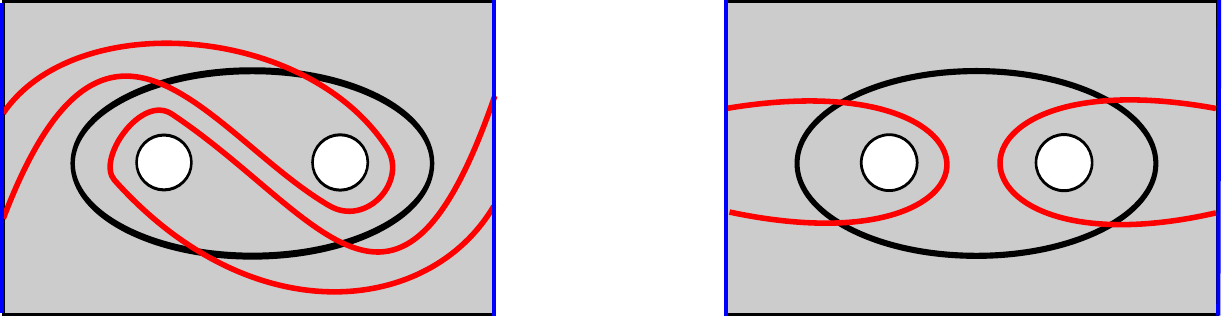}
\put(30, 82.5){$0$}
\put(45, 40){$1$}
\put(95, 40){$2$}
\put(30, 3){$3$}

\put(240, 82.5){$0$}
\put(240, 3){$3$}
\put(305, 40){$1$}
\put(255, 40){$2$}
\end{overpic}
\caption{Another choice of the arc for the curve $B''_2$ shown in the left. The right one is obtained from the left one by a diffeomorphism moving the hole $c_1$ to the right and the hole $c_2$ to the left.}
\label{fig:open_book_and_arc_a_shown0}
\end{figure}

If the Type I arc encloses the hole $c_2$ with the left edge of the rectangle, and the Type II arc encloses the hole $c_1$ with the right edge of the rectangle, then some power of  $\tau_{B_1}$ will send $B''_2$ to be the one which is formed by the two arcs shown in the left of Figure~\ref{fig:open_book_and_arc_a_shown0}. So there are at most two choices for the curve $B''_2$ up to conjugation. Hence, there are at most two different factorizations of $\tau_{B_1} \tau_{B_2}$ up to a global conjugation.
\end{proof}

\begin{lemma}
\label{general_uniqueness_of_factorization}
Let $\Sigma$ be the surface in Figure~\ref{fig:p14} with  $k=1$ or $2$, and  $n=2$.  Suppose $p=q=1$.
Then $\tau_{B_1}\tau_{B_2}\in Map(\mathbb{D}_4,\partial  \mathbb{D}_4)$ has at most two  positive factorizations up to a global conjugation.
\end{lemma}

\begin{proof}
If $k=1$ and $n=2$, then there is no hole which is enclosed by $B_2$ but not by $B_1$, and there is one hole which is enclosed by $B_1$ but not by $B_2$.  If $k=n=2$, then there is no hole which is enclosed by $B_1$ but not by $B_2$, and there is one hole which is enclosed by $B_2$ but not by $B_1$.  By symmetry, we can only prove for the first case.

By Lemma~\ref{lemma_combinatorial}, any positive factorization of $\tau_{B_1} \tau_{B_2}$ is $\tau_{B'_1} \tau_{B'_2}$ up to a global conjugation, where $B'_1$ and $B'_2$ enclose the same set of holes as $B_1$ and $B_2$, respectively.

The curves $B_1$ and $B_2$ fill the surface $\mathbb{D}_4$. Following exactly the same argument we get that $B'_1$ and $B'_2$ intersect in exactly $4$ points. Now we conjugate  $\tau_{B'_1} \tau_{B'_2}$ by a diffeomorphism which takes curve $B'_1$ to $B_1$. This will change $\tau_{B'_1} \tau_{B'_2}$ to $\tau_{B_1} \tau_{B''_2}$, where $B''_2$ is a curve which encloses the same set of holes as $B_2$.

Let $\gamma$ be an arc connecting holes $c_4$ and $c_0$, see the left of Figure~\ref{fig:open_book_and_arc_a_shown1}. Then, by the proof of Lemma~\ref{lemma_uniqueness_of_positive_factorization}, $B''_2$ intersects $\gamma$ in $2$ points minimally. We cut the surface $\mathbb{D}_4$ open along the arc $\gamma$, and think of the resulted surface with the outer boundary drawn as a rectangle.  See the right of Figure~\ref{fig:open_book_and_arc_a_shown1}. So $B''_2$ becomes two arcs one of which has two endpoints belong to the left edge and encloses one hole of $c_1$ and $c_3$, and the other of which  has two endpoints belong to the right edge and encloses the other hole of $c_1$ and $c_3$.

Exactly as in the proof of Lemma~\ref{lemma_uniqueness_of_positive_factorization}, we get that up to a diffeomorphism preserving orientation and commuting with $\tau_{B_1}$, there are at most two choices for the curve $B''_2$. Hence, there are at most two different positive factorizations of $\tau_{B_1} \tau_{B_2}$ up to  a global conjugation.
\end{proof}

\begin{figure}[htb]
\begin{overpic}
{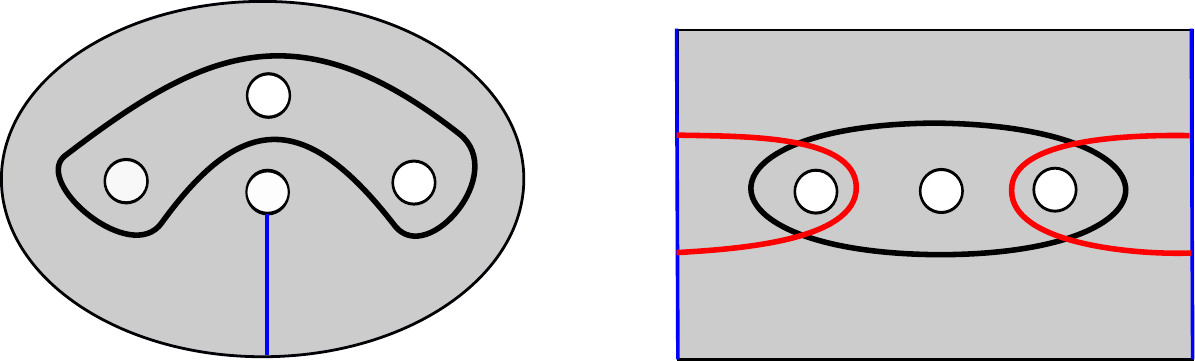}
\put(130, 10){$0$}
\put(34, 48){$1$}
\put(75, 73){$2$}
\put(117, 48){$3$}
\put(74, 45){$4$}
\put(80, 20){$\gamma$}
\put(300, 85){$0$}
\put(300, 5){$4$}
\put(234, 45){$1$}
\put(270, 45){$2$}
\put(303, 46){$3$}
\end{overpic}
\caption{Arc $\gamma$ along which the surface is cut open is shown in the left picture. The right picture shows the cut open surface.}
\label{fig:open_book_and_arc_a_shown1}
\end{figure}
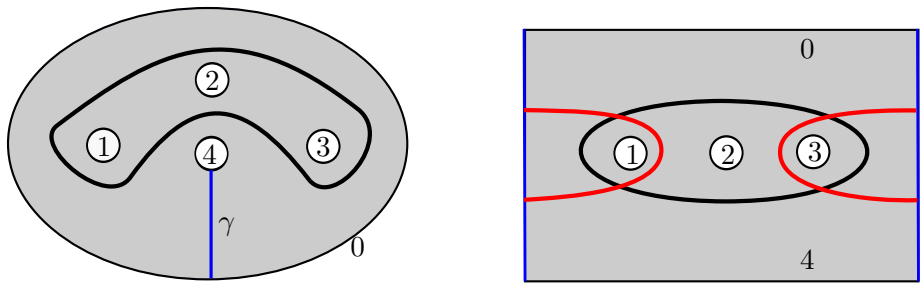

\begin{lemma}
\label{more_general_uniqueness_of_factorization}
Let $\Sigma$ be the surface in Figure~\ref{fig:p14} with  $k=2$ and  $n=3$. Suppose $p=q=1$.
Then $\tau_{B_1}\tau_{B_2}\in Map(\mathbb{D}_5,\partial\mathbb{D}_5)$ has at most two positive factorizations up to a global conjugation.
\end{lemma}

\begin{proof}
Since $k=2$ and  $n=3$, there is one hole which is enclosed by $B_2$ but not by $B_1$, and one hole which is enclosed by $B_1$ but not by $B_2$.

\begin{figure}[htb]
\begin{overpic}
{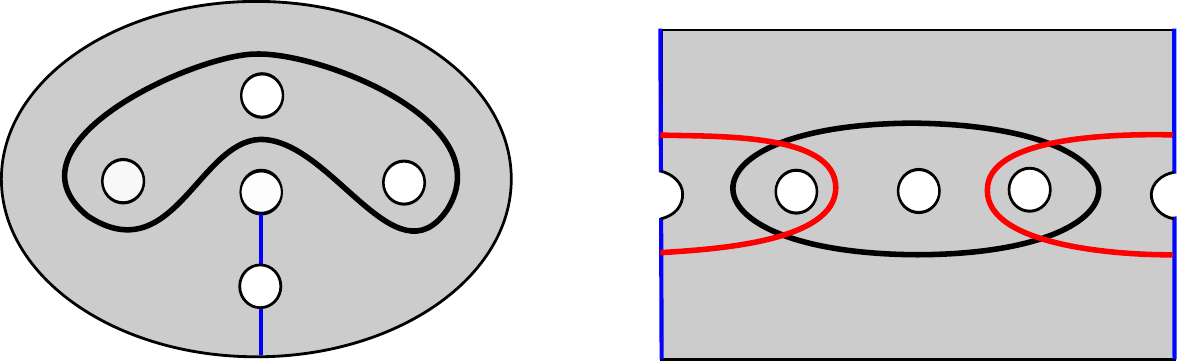}
\put(130, 10){$0$}
\put(73,18){$1$}
\put(33, 48){$2$}
\put(73, 73){$3$}
\put(115, 48){$4$}
\put(73, 45){$5$}
\put(77, 7){$\gamma_2$}
\put(77, 33){$\gamma_1$}
\put(300, 85){$0$}
\put(300, 5){$5$}
\put(190, 45){$1$}
\put(227, 45){$2$}
\put(263, 45){$3$}
\put(294, 45){$4$}
\put(335, 45){$1$}
\put(178, 10){$\gamma_{1}^{1}$}
\put(342, 10){$\gamma_{1}^{2}$}
\put(178, 80){$\gamma_{2}^{1}$}
\put(342, 80){$\gamma_{2}^{2}$}
\end{overpic}
\caption{Arcs $\gamma_1$ and $\gamma_2$ along which the surface is cut open is shown in the left picture. The right picture shows the cut open surface, where $\gamma_{i}^{1}$ and $\gamma_{i}^{2}$ are two copies of $\gamma_{i}$, for $i=1,2$.}
\label{fig:open_book_and_arc_a_shown2}
\end{figure}
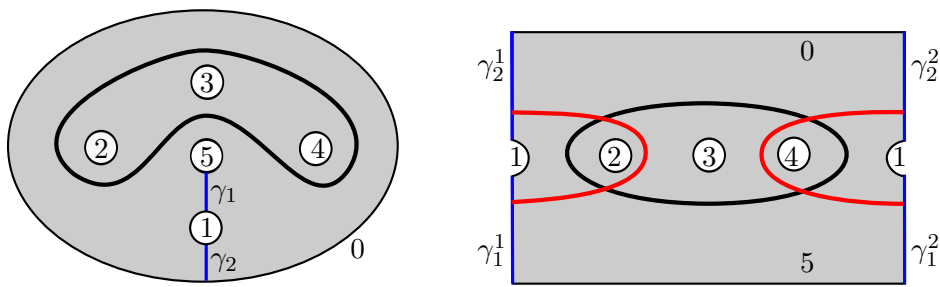

By Lemma~\ref{lemma_combinatorial}, any positive factorization of $\tau_{B_1} \tau_{B_2}$ is $\tau_{B'_1} \tau_{B'_2}$ up to a global conjugation, where $B'_1$ and $B'_2$ enclose the same set of holes as $B_1$ and $B_2$, respectively.

The curves $B_1$ and $B_2$ fill the surface $\mathbb{D}_5$. Following exactly the same argument as in the proof of Lemma~\ref{lemma_uniqueness_of_positive_factorization}, we get that $B'_1$ and $B'_2$ intersect in exactly $4$ points. Now we conjugate  $\tau_{B'_1} \tau_{B'_2}$ by a diffeomorphism which takes curve $B'_1$ to $B_1$. This will change $\tau_{B'_1} \tau_{B'_2}$ to $\tau_{B_1} \tau_{B''_2}$, where $B''_2$ is a curve which encloses the same set of holes as $B_2$.

If we fill each hole of $\mathbb{D}_5$, including the outer boundary, by a disk with a marked point, then we get a $2$-sphere with $6$ marked points.  The two curves $B_1$ and $B''_2$ give a cell decomposition of the $2$-sphere. It has four vertices, eight edges and six 2-cells. Each 2-cell contains a boundary component of $\mathbb{D}_5$.  There are four 2-cells which are bigons containing $c_0$, $c_2$, $c_4$, and $c_5$, respectively. There are two 2-cells which are squares containing $c_1$ and $c_3$, respectively. Each square has exactly one common edge with each of the four bigons. So there is a properly embedded arc $\gamma'_1$  in $S$ which connects holes $c_1$ and $c_5$, has exactly one intersection point with the common edge of the square containing $c_1$ and the bigon containing $c_5$, and is disjoint with $B_1$. There is a properly embedded arc $\gamma'_2$ in $S$ which connects holes $c_1$ and $c_0$, has exactly one intersection point with the common edge of the square containing $c_1$ and the bigon containing $c_0$, and is disjoint with $B_1$. Note that $\gamma'_1$ and $\gamma'_2$ correspond to two coedges of the cell decomposition. It is easy to make sure that we can choose the arcs $\gamma'_1$ and $\gamma'_2$ to be disjoint.

There is a diffeomorphism of $\mathbb{D}_5$ which takes $\gamma'_1$ and $\gamma'_2$ to $\gamma_1$ and $\gamma_2$, respectively, where $\gamma_1$ and $\gamma_2$ are as shown in the left of Figure~\ref{fig:open_book_and_arc_a_shown2} and keeps $B_1$ invariant. Such a diffeomorphism exists because the arcs $\gamma'_1$ and $\gamma'_2$ are chosen to be disjoint from the curve $B_1$. We denote the image of $B''_2$ under this diffeomorphism by $B''_2$ still. Then $I(B''_2, \gamma_i)=1$ for $i=1,2$.

We cut the surface $\mathbb{D}_5$ open along  $\gamma_1\cup\gamma_2$, and think of the resulting surface with the outer boundary drawn as a rectangle. To avoid confusion with other terminology used, we will denote this cut open surface by $\mathcal{R}$.  See the right of Figure~\ref{fig:open_book_and_arc_a_shown2}. The curve $B''_2$ is cut open into two arcs. One of them has both endpoints in the left edge of $\mathcal{R}$, and the other of them has both endpoints in the right edge of $\mathcal{R}$. Since $B_1$ is parallel to the outer boundary of $\mathcal{R}$, each of the two arcs are either disjoint with $B_1$ or has exactly two intersection points with $B_1$. Since $I(B_1, B''_2)=4$, both of the two arcs have exactly two intersection points with $B_1$. One of them encloses one hole of  $c_2$ and $c_4$ with a subarc of the left edge of $\mathcal{R}$. The other of them encloses the other hole of $c_2$ and $c_4$ with a subarc of the right edge of $\mathcal{R}$.

Just as in the proofs of Lemma~\ref{lemma_uniqueness_of_positive_factorization} and Lemma~\ref{general_uniqueness_of_factorization}, we get that there are at most two choices for $B''_2$. Hence,  $\tau_{B_1} \tau_{B_2}$ has at most two positive factorizations up to a global conjugation.
\end{proof}

\subsection{\textbf{Stein fillings of certain planar open books}}

\noindent Now we go back to the proof of Theorem~\ref{theorem_factorization_of_mapping_class}.

\begin{figure}[htb]
\begin{overpic}
{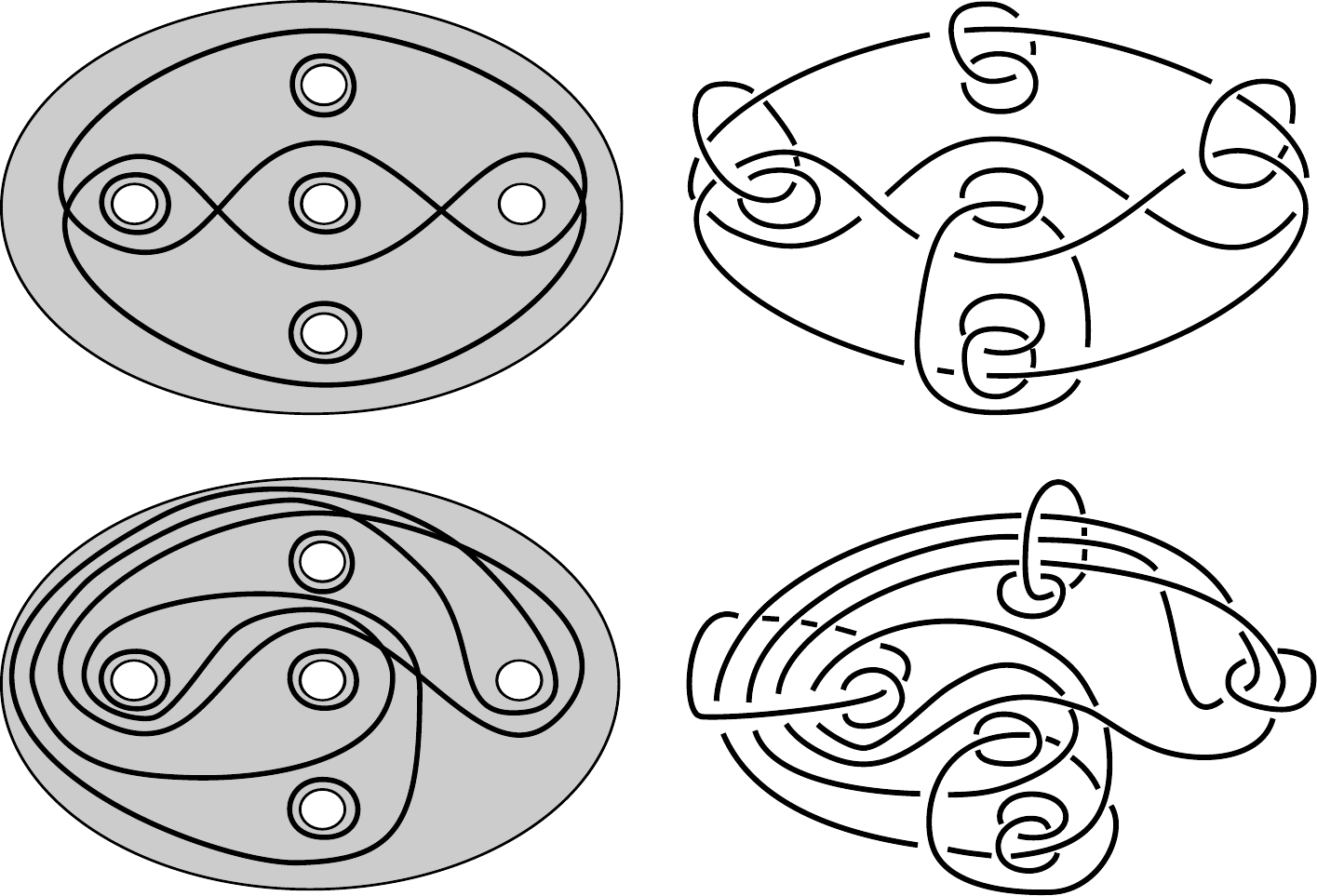}

\put(10, 20){$0$}
\put(98, 22){$1$}
\put(40, 62){$2$}
\put(97, 100){$3$}
\put(160, 62){$4$}
\put(97, 62){$5$}

\put(10, 168){$0$}
\put(98, 170){$1$}
\put(40, 210){$2$}
\put(97, 247){$3$}
\put(160, 210){$4$}
\put(97, 210){$5$}

\put(300, 274){$\bullet$}
\put(390, 250){$\bullet$}
\put(213, 241){$\bullet$}
\put(310, 153){$\bullet$}
\put(321, 149){$\bullet$}

\put(324, 125){$\bullet$}
\put(405, 67){$\bullet$}
\put(211.5, 60){$\bullet$}
\put(318, 5){$\bullet$}
\put(319, -2){$\bullet$}

\put(120, 169){$B''_2$}
\put(67, 252){$B_1$}
\put(127, 20){$B''_2$}
\put(67, 103){$B_1$}

\end{overpic}
\caption{The left two figures indicate the two choices of the curve $B''_2$. The upper left figure is $(\Sigma, \tau_{1}\tau_{2}\tau_{3}\tau_{5}\tau_{B_1}\tau_{B''_2})$. The lower left figure is $(\Sigma, \tau_{1}\tau_{2}\tau_{3}\tau_{5}\tau_{B_1}\tau_{B''_2})$ which is conjugate to $(\Sigma, \tau_{1}\tau_{2}\tau_{3}\tau_{5}\tau_{B''_2}\tau_{B_1})$. The right two figures are their corresponding Kirby diagrams for the Stein filling, where each dotted circle is a $4$-dimensional 1-handle, and all other circles have surgery coefficients $-1$. These two Kirby diagrams denote two diffeomorphic $4$-manifolds.}
\label{fig:Kirby diagrams}
\end{figure}
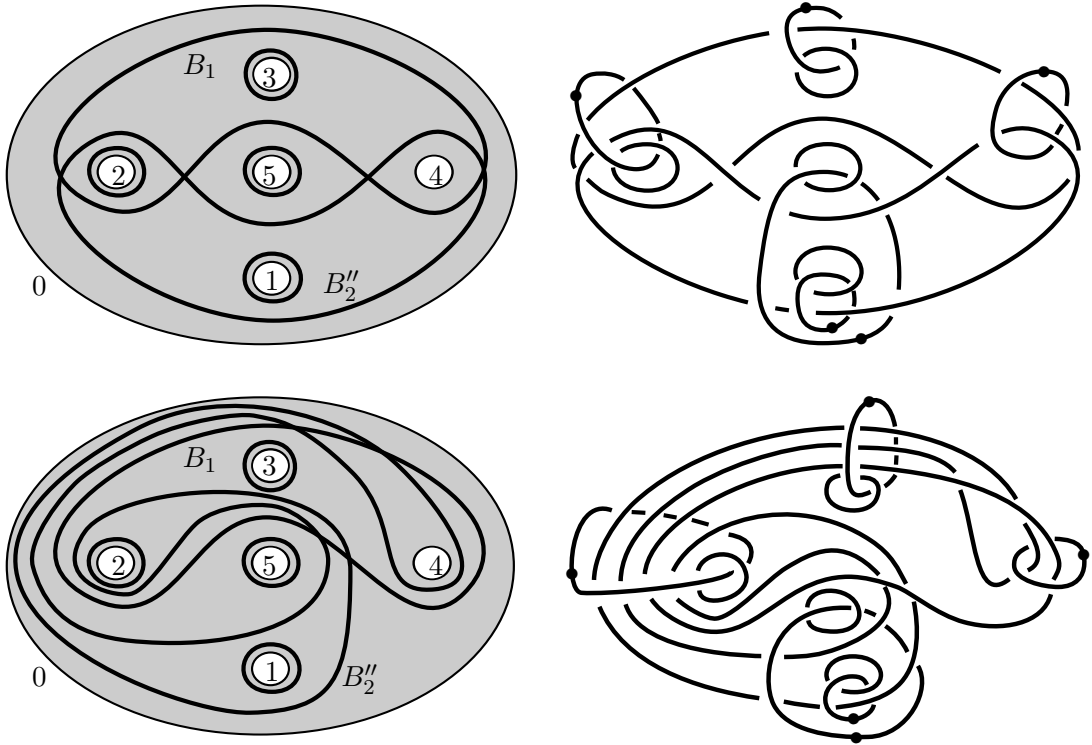

\begin{proof}

Suppose $a_{1},\dots , a_{k-1}$ are $k-1$ properly embedded pairwise disjoint arcs in $\Sigma$ which satisfy that: 1) $a_i$ connects the boundary components $c_i$ and $c_{i+1}$, 2) $a_i$ is disjoint with $B_1$ and $B_2$. Suppose $b_{k+q},\dots , b_{n+q-2}$ are $n-k-1$ properly embedded pairwise disjoint arcs in $\Sigma$ which satisfy that: 1) $b_i$ connects the boundary components $c_i$ and $c_{i+1}$, 2) $b_i$ is disjoint with $B_1$ and $B_2$.

If $p=q=1$, then by Lemma~\ref{lemma_combinatorial}, any other positive factorization of $\Phi$ must be the product of $\tau_{1}^{m_1}$, $\tau_{2}^{m_2}$, $\ldots$, $\tau_{n}^{m_{n}}$, $\tau_{n+2}^{m_{n+2}}$, $\tau_{B'_1}$ and $\tau_{B'_2}$, where the curves $B'_1$ and $B'_2$ enclose the same set of holes as curves $B_1$ and $B_2$, respectively. In particular, either $\tau_{B_1} \tau_{B_2}=\tau_{B'_1} \tau_{B'_2}$ or $\tau_{B_1} \tau_{B_2}=\tau_{B'_2} \tau_{B'_1}$. Without loss of generality, we assume that $\tau_{B_1} \tau_{B_2}=\tau_{B'_1} \tau_{B'_2}$.

Since the curves $B_1$ and $B_2$ do not intersect any of the arcs $a_i$, the diffeomorphism $\tau_{B_1} \tau_{B_2}$ does not move them. It follows that the diffeomorphism $\tau_{B'_1}\tau_{B'_2}$ does not move any of the arcs $a_i$ as well. We claim that the curves $B'_1$ and $B'_2$ do not intersect any of the arcs $a_i$. Suppose one of the curves $B'_1, B'_2$ intersects the arc $a_i$ for some $i$. Without loss of generality we can assume that curve to be $B'_2$. In this case, since any positive Dehn twist is right veering(see \cite{HKM_Right_Veering_1}), the diffeomorphism $\tau_{B'_2}$ will move the arc $a_i$ strictly to the right.  Hence, $\tau_{B'_1}$ should move the arc $a_i$ strictly to the left, which is impossible.  Similarly, the curves $B'_1$ and $B'_2$ do not intersect any of the arcs $b_i$.

Since the arcs $a_i$ and $b_i$ are not moved by the diffeomorphism $\tau_{B'_1} \tau_{B'_2}$, we can cut along arcs $a_i$ and $b_i$. If we do that we are left with the surface $\mathbb{D}_3$, $\mathbb{D}_4$ or $\mathbb{D}_5$, depending on $k$ and $n$. We still denote the curves by $B_1,B_2, B'_1, B'_2$ in this new surface. As before, there is a diffeomorphism which sends $B'_1$ to $B_1$, and $B'_2$ to $B''_2$. Now from Lemmas~\ref{lemma_uniqueness_of_positive_factorization}, ~\ref{general_uniqueness_of_factorization} and  ~\ref{more_general_uniqueness_of_factorization} it follows that there are at most two factorizations, up to a global conjugation, of the monodromy given by two different choices for the curve $B''_2$.

From the Kirby diagrams, see  Figure~\ref{fig:Kirby diagrams} for an example, we know that for both of these two choices of the curve $B''_2$, the manifold supported by the open book decomposition $$(\Sigma,\tau_{1}^{m_1}\tau_{2}^{m_2}\ldots \tau_{n}^{m_{n}}\tau_{n+2}^{m_{n+2}}\tau_{B_1}\tau_{B''_2})$$ is diffeomorphic to the original oriented 3-manifold, and their corresponding Stein fillings are diffeomorphic.  According to Theorem~\ref{Wendl_Planar_open_books}, Theorem~\ref{theorem_factorization_of_mapping_class} holds in this special case.

Now we are left to prove the general case.  Any other  factorization of $\Phi$ is $$\tau_{1}^{m_1}\tau_{2}^{m_2}\ldots \tau_{n+q-1}^{m_{n+q-1}}\tau_{n+q+1}^{m_{n+q+1}}\ldots\tau_{n+p+q}^{m_{n+p+q}}\tau_{B'_1}\tau_{B'_2}$$ up to a global conjugation, where $B'_1$ and $B'_2$ enclose the same boundary components as $B_1$ and $B_2$, respectively. In particular, either $\tau_{B_1} \tau_{B_2}=\tau_{B'_1} \tau_{B'_2}$ or $\tau_{B_1} \tau_{B_2}=\tau_{B'_2} \tau_{B'_1}$. Without loss of generality, we assume that $\tau_{B_1} \tau_{B_2}=\tau_{B'_1} \tau_{B'_2}$.

Suppose $u_{n+q+1},\dots , u_{n+q+p-1}$ are $p-1$ properly embedded pairwise disjoint arcs in $\Sigma$ which satisfy that: 1) $u_i$ connects the boundary components $c_i$ and $c_{i+1}$, 2) $u_i$ is disjoint with $B_1$ and $B_2$. Suppose $v_{k},\dots , v_{k+q-2}$ are $q-1$ properly embedded pairwise disjoint arcs in $\Sigma$ which satisfy that: 1) $v_i$ connects the boundary components $c_i$ and $c_{i+1}$, 2) $v_i$ is disjoint with $B_1$ and $B_2$.

Note that the diffeomorphism $\tau_{B_1} \tau_{B_2}$ does not move any of the arcs $u_{n+q+1},\dots,$  $u_{n+q+p-1},  v_{k}$, $\dots, v_{k+q-2}$ and so $\tau_{B'_1} \tau_{B'_2}$ does not move these arcs either. By the same argument as in previous paragraph, the curves $B'_1$ and $B'_2$ do not intersect arcs $u_{n+q+1},\dots,$  $u_{n+q+p-1},  v_{k}$, $\dots, v_{k+q-2}$.

Hence, to factorize $\Phi$ we need to specify curves $B'_1$ and $B'_2$ in the complement of arcs $u_{n+q+1},\dots,$  $u_{n+q+p-1},  v_{k}$, $\dots, v_{k+q-2}$.  So we cut the surface along arcs $u_{n+q+1}$, $\dots$, $u_{n+q+p-1}$, $v_{k}$, $\dots$, $v_{k+q-2}$, and what left is a planar surface with $n+3$ boundary components. Also, there is a diffeomorphism which sends $B'_1$ to $B_1$, and $B'_2$ to $B''_2$.

By Lemma~\ref{lemma_uniqueness_of_positive_factorization}, Lemma~\ref{general_uniqueness_of_factorization} and Lemma~\ref{more_general_uniqueness_of_factorization}, there are at most two factorizations, up to a global conjugation, of the monodromy given by two different choices for the curve $B''_2$. Considering the Kirby diagrams, we know that the two Stein fillings are diffeomorphic. This finishes the proof by Theorem~\ref{Wendl_Planar_open_books}. \end{proof}

\subsection{\textbf{Proofs of the corollaries}}

\begin{proof}[Proof of Corollary~\ref{twist_knot_family}]

Let $L'$ be a Legendrian twist knot $K_{-2p}$ with Thurston-Bennequin invariant $-n$ and rotation number $n-2k+1$. According to~\cite{Li_Wang_support_genera_of_legendrian_knots}, we can embed the Legendrian link in Figure~\ref{fig:Legendrian_surgery_diagram_for_infinite_family1} into a page of an embedded open book supporting $(S^3, \xi_{std})$. See Figure~\ref{fig:embedded_open_book}. In the figure, $L'$ denotes a Legendrian twist knot $K_{-2p}$ and $U_1,\dots,U_{k-1},U_{k+1},\dots,U_{n}$ denote Legendrian unknots with $tb=-1$ and $rot = 0$  embedded in the page of an open book decomposition supporting $(S^3,\xi_{std})$.

\begin{center}
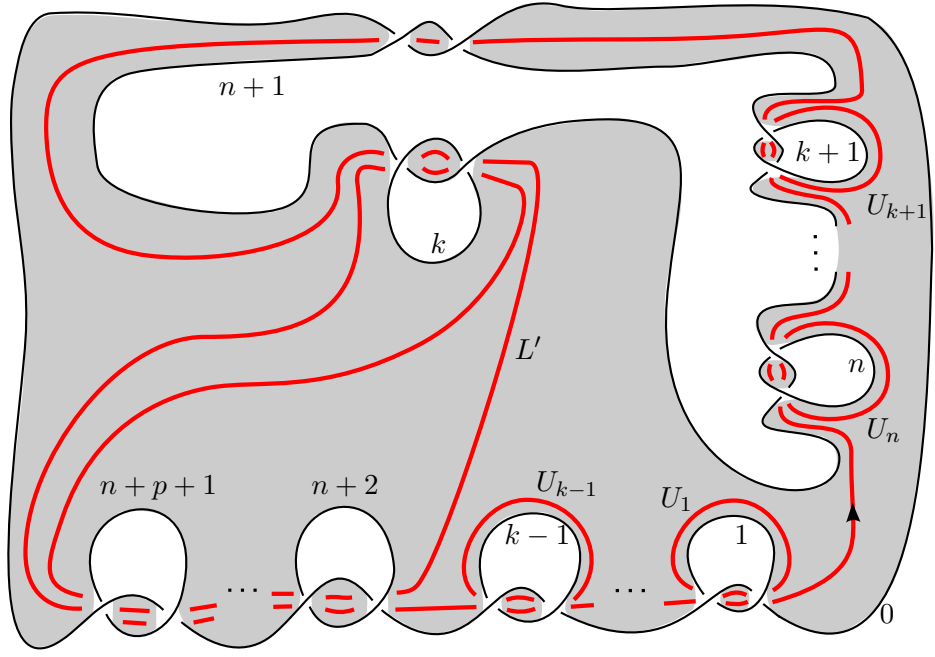
\begin{figure}[htb]
\begin{overpic}
{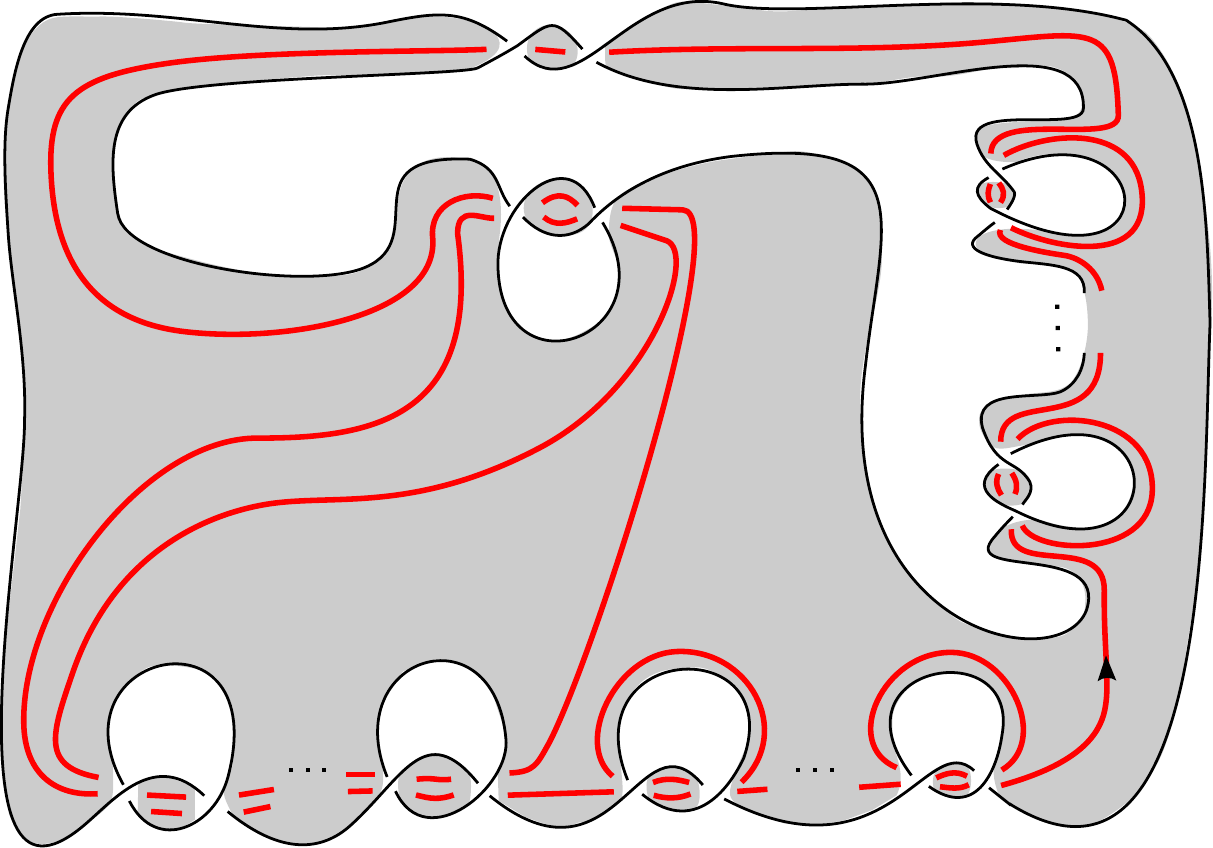}
\put(330, 10){$0$}
\put(318, 105){$n$}
\put(298, 185){$k+1$}
\put(275, 40){$1$}
\put(188, 40){$k-1$}
\put(115, 59){$n+2$}
\put(35, 59){$n+p+1$}
\put(160, 150){$k$}
\put(80, 210){$n+1$}
\put(192,110){$L'$}
\put(247,54){$U_1$}
\put(200,60){$U_{k-1}$}
\put(325,165){$U_{k+1}$}
\put(325,80){$U_n$}
\end{overpic}
\caption{An embedded open book decomposition supporting $(S^3, \xi_{std})$ with a twist knot $K_{-2p}$ shown in the figure by $L'$. Curves $U_1,\dots,U_{k-1},U_{k+1},\dots,U_{n}$ denote Legendrian unknots embedded in the page of open book decomposition. The boundary components of the page are labelled by $0$, $1$, $\cdots$, $n+p+1$. The shaded region in the figure shows the page of the open book decomposition. This open book is obtained from the open book in Figure 3 in~\cite{Li_Wang_support_genera_of_legendrian_knots} by plumbing some positive Hopf bands.   }
\label{fig:embedded_open_book}
\end{figure}
\end{center}

\begin{lemma}
\label{embed_to_abstract}
The embedded open book supporting $(S^3, \xi_{std})$ depicted in Figure~\ref{fig:embedded_open_book} can be transformed into an abstract version $(\Sigma, \phi)$ shown in Figure~\ref{fig:p14}, with $q=1$ and $$\phi=\tau_{1}\tau_{2}\ldots \tau_{k-1}\tau_{k}\tau_{k+1}\ldots\tau_{n}\tau_{n+2}\ldots\tau_{n+p+1}\tau_{B_1}.$$
\end{lemma}

\begin{proof} The embedded open book depicted in Figure~\ref{fig:embedded_open_book} can be obtained from the open book supporting $(S^3,\xi_{std})$, whose binding is an unknot, by plumbing finitely many positive Hopf bands successively.   It is easy to see that an embedded open book with an unknot as binding can be transformed to an abstract version whose page is a disk and monodromy is the identity.   Recall that plumbing a positive Hopf band either increases the genus of the page by $1$ or increases the number of boundary components by $1$, depending on the way in which the $1$-handle is attached. See~\cite{Etnyre_Lectures_on_open_book_decompositions} for details. In this proof, each plumbing is done so that we increase only the number of boundary components. With this set up, one can plumb the positive Hopf bands in the order that creates the boundary components $c_{1}$, $\ldots$, $c_{k-1}$, $c_{k}$, $c_{k+1}$, $\ldots$, $c_{n}$, $c_{n+2}$, $\ldots$, $c_{n+p+1}$, and $c_{n+1}$ one after another.

Suppose an embedded open book $(\partial S_1, \pi_1)$ with page $S_1$ supporting $(S^3, \xi_{std})$ is transformed into an abstract version $(\Sigma_1, \phi)$ through a diffeomorphism from $S_1$ to $\Sigma_1$, and an embedded open book $(\partial S_2, \pi_2)$ with page $S_2$ is obtained from $(\partial S_1, \pi_1)$ by plumbing a positive Hopf band along a properly embedded arc $a\subset S_1$.  By the definition of plumbing, the positive Hopf band is attached at the two end points of the arc $a$, so that the union of arc $a$ and the core of the positive Hopf band is a simple closed curve which can be pushed away from the page $S_1$ and does not link with $\partial S_1$. Note that the isotopy class of the resulted embedded open book $(\partial S_2, \pi_2)$ does not depend on the choice of the interior of $a$. We denote by $a$, still, the properly embedded arc in the page of abstract open book $(\Sigma_1,\phi)$ which is the image of the arc $a$ in $S_1$. Then we can transform the embedded open book $(\partial S_2, \pi_2)$ into an abstract version $(\Sigma_2, \phi\cdot\tau_{c})$, where $\Sigma_2$ is obtained from $\Sigma_1$ by adding a $1$-handle at the two end points of the arc $a$, and the simple closed curve $c\subset \Sigma_2$, called a \emph{stabilizing curve}, is the union of the arc $a\subset \Sigma_1$ and the core of the attached $1$-handle. The abstract open book $(\Sigma_2, \phi\cdot\tau_{c})$ is a positive stabilization of $(\Sigma_1, \phi)$.

It is easy to draw the arcs along which the positive Hopf bands are plumbed to obtain the embedded open book in Figure~\ref{fig:embedded_open_book}.  Following the local transformation process in the order of plumbing mentioned above, we obtain an abstract version of the open book decomposition in Figure~\ref{fig:embedded_open_book} as shown in Figure~\ref{fig:abstract_open_book_twist_knot}, with $L'$ being a Legendrian twist knot $K_{-2p}$ in $(S^3,\xi_{std})$. By definition, each positive Hopf band corresponds to a stabilizing curve. For $i=1,\cdots, n, n+2, \cdots, n+p+1$, there is a stabilizing curve for open book decomposition supporting $(S^3,\xi_{std})$ which is parallel to the $i$-th boundary component. Moreover, the curve $\alpha$ is also a stabilizing curve for open book decomposition supporting $(S^3,\xi_{std})$ which corresponds to the top positive Hopf band in Figure~\ref{fig:embedded_open_book} that creates the boundary component $c_{n+1}$. 

It is easy to see that there is a diffeomorphism taking the surface in Figure~\ref{fig:abstract_open_book_twist_knot} to the one in Figure~\ref{fig:p14} when $q=1$. Under this diffeomorphism the curves $L', \alpha$ in Figure~\ref{fig:abstract_open_book_twist_knot} become the curves $B_2, B_1$ in Figure~\ref{fig:p14}.
\end{proof}

\begin{center}
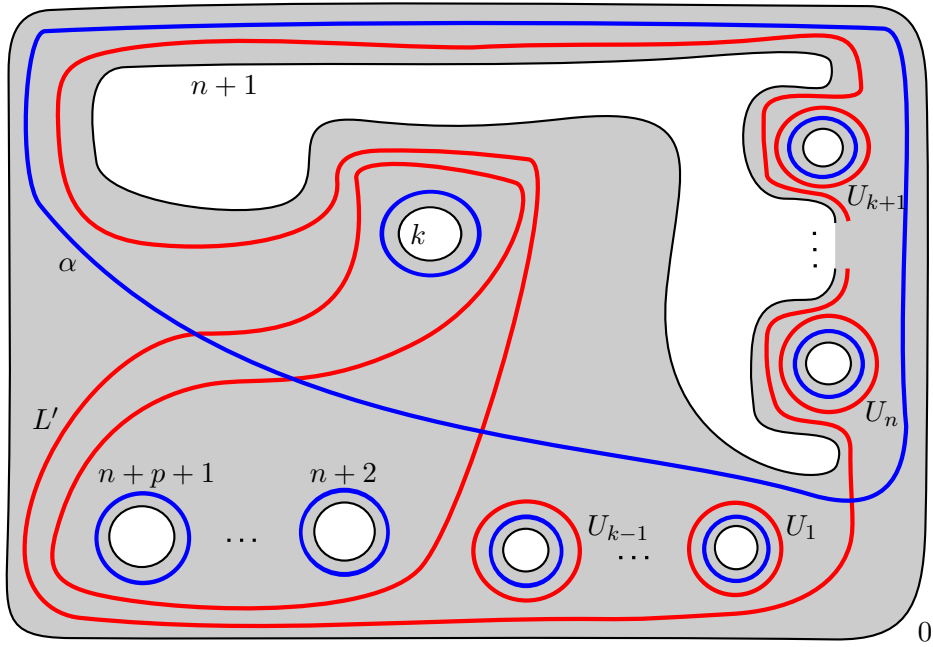
\begin{figure}[htb]
\begin{overpic}
{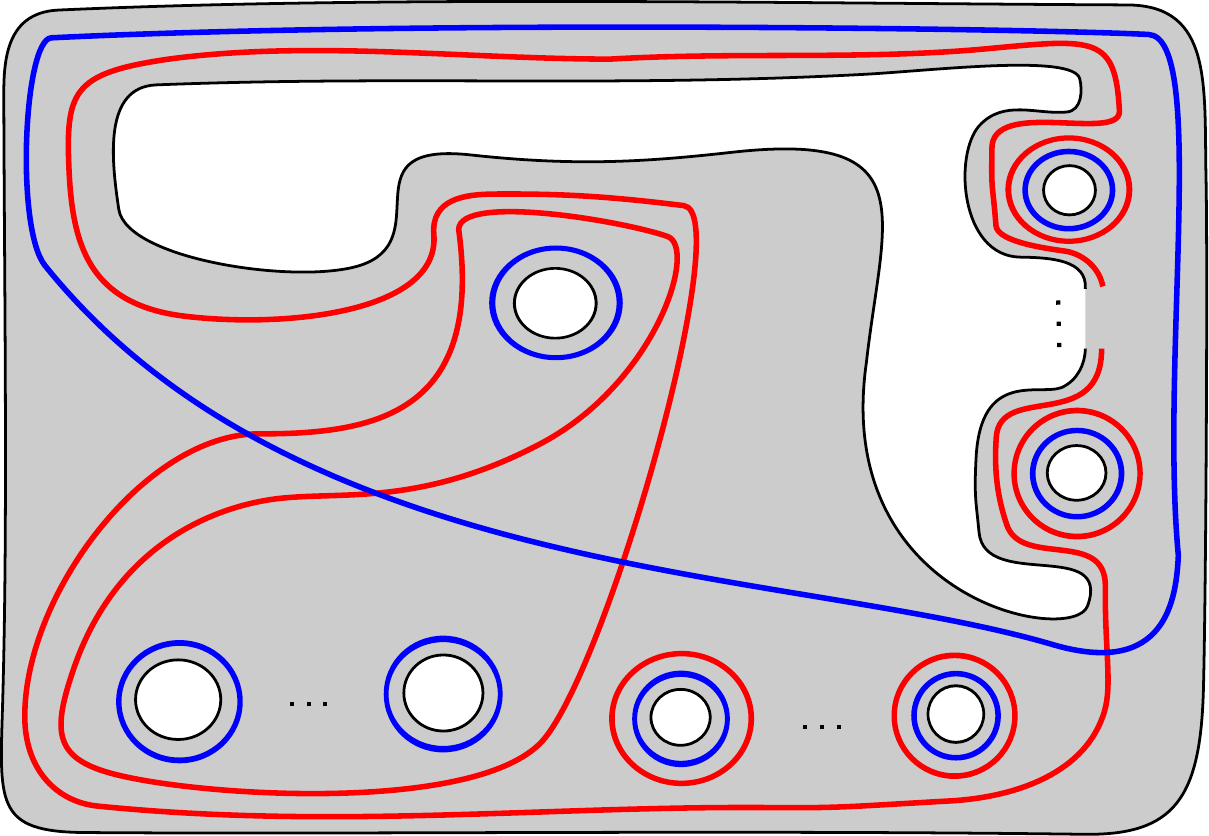}
\put(345, 0){$0$}
\put(153, 150){$k$}
\put(115, 60){$n+2$}
\put(35, 60){$n+p+1$}
\put(70, 207){$n+1$}
\put(10,80){$L'$}
\put(295,40){$U_1$}
\put(220,40){$U_{k-1}$}
\put(318,165){$U_{k+1}$}
\put(325,83){$U_n$}
\put(20,140){$\alpha$}
\end{overpic}
\caption{An abstract open book decomposition supporting $(S^3, \xi_{std})$ with a twist knot $K_{-2p}$ shown in the figure by $L'$. Curves $U_1,\dots,U_{k-1},U_{k+1},\dots,U_{n}$ denote Legendrian unknots embedded in the page of open book decomposition. The boundary components of the page are labelled by $0$, $1$, $\cdots$, $n+p+1$. In particular, the boundary component encircled by $U_i$ is labelled by $i$ for $i=1, \cdots, k-1, k+1, \cdots, n$.}
\label{fig:abstract_open_book_twist_knot}
\end{figure}
\end{center}

So the contact structure $(M',\xi')$ is supported by the open book $(\Sigma, \Phi)$ with $q=1$ and $$\Phi=\tau_{1}^{m_1}\tau_{2}^{m_2}\ldots \tau_{k-1}^{m_{k-1}}\tau_{k}\tau_{k+1}^{m_{k+1}}\ldots\tau_{n}^{m_{n}}\tau_{n+2}\ldots\tau_{n+p+1}\tau_{B_1}\tau_{B_2},$$
where $m_{i}\geq1$ for $i=1,\dots, k-1, k+1,\ldots, n$.

By Theorem~\ref{theorem_factorization_of_mapping_class},  $(M', \xi')$ admits a unique Stein filling up to diffeomorphism.
\end{proof}

\begin{proof}[Proof of Corollary~\ref{2_bridge_family}]

The open book $$(\Sigma, \tau_{1}\tau_{2}\ldots \tau_{n+q-1}\tau_{n+q+1}\ldots\tau_{n+p+q}\tau_{B_1})$$ corresponds to $(S^3, \xi_{std})$. If we transform it to the embedded version which is similar to Figure~\ref{fig:embedded_open_book}, then $B_2$ can be realized as a Legendrian $2$-bridge knot  $B(p,q)$ topologically shown in Figure~\ref{fig:2-bridge_knots}. By Lemma~\ref{lemma:knot_stabilization}, it is the result of $n-k$ positive stabilizations and $k-1$ negative stabilizations of a Legendrian 2-bridge knot $B(p,q)$ with Thurston-Bennequin invariant $-1$ and rotation number $0$. So by Theorem~\ref{theorem_factorization_of_mapping_class}, the contact 3-manifold which is obtained by Legendrian surgery on $(S^3, \xi_{std})$ along the Legendrian $B_2$ admits a unique Stein filling up to diffeomorphism.
\end{proof}

\subsection{\textbf{Uniqueness of certain Stein fillings up to symplectic deformation equivalence}}

\begin{proof}[Proof of Theorem~\ref{trefoil_family}]

Since $L$ is a Legendrian twist knot $K_{-2p}$ with Thurston-Bennequin invariant $-1$ and rotation number $0$, we can embed $S^{n-k}_{+}S^{k-1}_{-}(L)$ into a page of an embedded open book decomposition supporting $(S^3, \xi_{std})$ as in Figure~\ref{fig:embedded_open_book},  where the page is a compact planar surface with $n+p+2$ boundary components. By Lemma~\ref{embed_to_abstract}, we can transform this embedded open book decomposition into an abstract version $(\Sigma, \phi)$ with  $q=1$ and $$\phi=\tau_{1}\tau_{2}\ldots \tau_{n}\tau_{n+2}\ldots \tau_{n+p+1}\tau_{B_1}.$$

The Legendrian surgery on $(S^3, \xi_{std})$ along the stabilization $S^{n-k}_{+}S^{k-1}_{-}(L)$ yields a contact structure $\xi_k$ on the 3-manifold $S^{3}_{-1-n}(K_{-2p})$, which is supported by the open book decomposition $(\Sigma, \Phi)$ with $q=1$ and $$\Phi=\tau_{1}\tau_{2}\ldots\tau_{n}\tau_{n+2}\ldots \tau_{n+p+1}\tau_{B_1}\tau_{B_2}.$$

By Lemma~\ref{lemma_combinatorial}, any positive factorization of $\Phi$ has to be the product of $\tau_{1}$, $\tau_{2}$, $\ldots$, $\tau_{n}$, $\tau_{n+2}$, $\ldots$,  $\tau_{n+p+1}$, $\tau_{B'_1}$, $\tau_{B'_2}$,  where $B'_1$ and $B'_2$ enclose the same set of holes as $B_1$ and $B_2$, respectively.

The open book decomposition $$(\Sigma, \tau_{1}\tau_{2}\ldots\tau_{n}\tau_{n+2}\ldots \tau_{n+p+1}\tau_{B'_1})$$ also supports  $(S^3, \xi_{std})$. We think of  $B'_2$ as a knot in $(S^3, \xi_{std})$.

We claim that  $B'_2$ is isotopic to the twist knot $K_{-2p}$. There is an element $f\in Map(\Sigma, \partial\Sigma)$ which sends $B'_1$ to $B_1$. We denote $f(B'_2)$ by $B''_{2}$.  According to the proof of Theorem~\ref{theorem_factorization_of_mapping_class}, the given monodromy $\Phi$ has at most two different  positive factorizations, up to a global conjugation, depending on the two choices for $B''_2$.  Both of the two choices of  $B''_2$ are isotopic to the twist knot $K_{-2p}$. So $B'_2$ is isotopic to the twist knot $K_{-2p}$.

We compute the page framing of $B'_{2}$ in the open book decomposition $$(\Sigma, \tau_{1}\tau_{2}\ldots\tau_{n}\tau_{n+2}\ldots \tau_{n+p+1}\tau_{B'_1}).$$ To this end, we compute the linking number of $B''_{2}$ and its push-off in the page of open book decomposition $$(\Sigma, \tau_{1}\tau_{2}\ldots\tau_{n}\tau_{n+2}\ldots \tau_{n+p+1}\tau_{B_1})$$ shown in Figure~\ref{fig:embedded_open_book}.  For both of the two choices of $B''_2$, it is routine to check that the linking numbers of $B''_{2}$ and its push-off in the page are $-n$. So $B'_{2}$ has page framing $-n$ with respect to the Seifert framing.

Since $B'_2$ is not null-homologous in $\Sigma$, we can Legendrian realize it.  According to the definition of open book decomposition,  we know that the Thurston-Bennequin invariant of $B'_2$ is the difference between the page framing and the Seifert framing, that is, $-n$.

Therefore, the Lefschetz fibration $X$ over $D^2$, with fiber $\Sigma$, corresponding to the positive factorization $\tau_{1}\tau_{2}\ldots\tau_{n}\tau_{n+2}\ldots \tau_{n+p+1}\tau_{B'_1}\tau_{B'_2}$ of  $\Phi$  is diffeomorphic to $D^4$, with its standard complex structure, and a 2-handle attached along a $(-1-n)$-framed twist knot $K_{-2p}$. Also, $X$ has a Stein structure that arises from the Legendrian surgery along the Legendrian realized $B'_2$. By a theorem of Eliashberg \cite{Eliashberg_Stein_handle_attachment}, we can extend the Stein structure uniquely to this new manifold. Since there is a unique Legendrian twist knot $ K_{-2p}$ with Thurston-Bennequin invariant $-n$ and rotation number $n-2k+1$, \cite{Etnyre_Ng_Vertesi_Twist_knots_classification}, we know that the only Legendrian twist knot $K_{-2p}$ that can produce $(S^{3}_{-1-n}(K_{-2p}), \xi_k)$ is $S^{n-k}_{+}S^{k-1}_{-}(L)$. This implies that all Stein structures on $X$ are symplectic deformation equivalent. So $X$ has a unique Stein structure up to symplectic deformation.

According to Theorem~\ref{Wendl_Planar_open_books}, every Stein filling of $(S^{3}_{-1-n}(K_{-2p}), \xi_k)$ is symplectic deformation equivalent to a Lefschetz fibration compatible with the given planar open book decomposition  $(\Sigma, \Phi)$.  So there is a unique Stein filling on $(S^{3}_{-1-n}(K_{-2p}), \xi_k)$, up to symplectic deformation.
\end{proof}


\begin{thebibliography}{10}


\bibitem{Akbulut_Ozbagci_Stein_Surface_Lefschetz_fibrations}
Selman Akbulut and Burak Ozbagci, \emph{Lefschetz fibrations on compact {S}tein surfaces}, Geom. Topol. \textbf{5}(2001), 319--334 (electronic), \MR{1825664 (2003a:57055)}.


\bibitem{Eliashberg_Filling_by_holomorphic_disks_and_applications}
Yakov Eliashberg, \emph{Filling by holomorphic discs and its applications}, from: ``Geometry of low-dimensional manifolds, 2 (Durham, 1989)", (SK Donaldson, CB Thomas, editors), London Math. Soc. Lecture Note Ser. 151, Cambridge Univ. Press(1990), 45-67. \MR{1171908 (93g:53060)}

\bibitem {Eliashberg_Stein_handle_attachment} Yakov Eliashberg,
     \emph{Topological characterization of {S}tein manifolds of dimension {$>2$}}, Internat. J. Math.  \textbf{1} (1990), no. 1, 29--46,\MR{1044658 (91k:32012)}.


\bibitem {Etnyre_Lectures_on_open_book_decompositions}John B. Etnyre
   \emph{Lectures on open book decompositions and contact structures}, Floer homology, gauge
theory, and low-dimensional topology, Clay Math. Proc., vol. 5, Amer. Math. Soc., Providence, RI, 2006,
pp. 103--141. \MR{2249250 (2007g:57042)},

\bibitem{Etnyre_Ng_Vertesi_Twist_knots_classification} John B. Etnyre, Lenhard  Ng, and Vera V\'{e}rtesi,
 \emph{Legendrian and Trasnverse twist knots},
 J. Eur. Math. Soc. (JEMS) \textbf{15} (2013), no. 3, 969--995.

\bibitem{Farb_Margalit_Primer} Benson Farb and Dan Margalit,
A primer on mapping class groups, Princeton Mathematical Series, vol. 49,
Princeton University Press, Princeton, NJ, 2012. \MR {2850125 (2012h:57032)}




\bibitem{FLP_Thurston's_work_on_surfaces}
A. Fathi, F. Laudenbach and V. Po{\'e}naru,
\emph{Travaux de {T}hurston sur les surfaces},
Ast\'erisque,
66,
S{\'e}minaire Orsay,
              With an English summary,
Soci\'et\'e Math\'ematique de France,
Paris,
1979,
284,
\MR {568308 (82m:57003)},



\bibitem{Giroux_Correspondence}
Emmanuel Giroux,
\emph{G\'eom\'etrie de contact: de la dimension trois vers les
              dimensions sup\'erieures},
Proceedings of the {I}nternational {C}ongress of
              {M}athematicians, {V}ol. {II} ({B}eijing, 2002),
405--414,
Higher Ed. Press,
Beijing, 2002,
\MR{1957051 (2004c:53144)}

\bibitem{Gompf_Stipsicz} Robert E. Gompf and Andr{\'a}s I. Stipsicz, \emph{{$4$}-manifolds and {K}irby calculus}, Graduate Studies in Mathematics, vol. 20, American Mathematical Society, Providence, RI, 1999. \MR{ 1707327 (2000h:57038)}

\bibitem{Grauert_Levi's_problem_and_imbedding_real_analytic_manifolds_1958}Hans Grauert, \emph{On Levi's problem and the imbedding of real-analytic manifolds}, Ann. of Math. (2) \textbf{68}
(1958), 460--472. \MR {0098847 (20 \#5299)}

\bibitem{HKM_Right_Veering_1}Ko Honda, William H. Kazez, and Gordana Mati{\'c}, \emph{Right-veering diffeomorphisms of compact surfaces with boundary}, Invent. Math. \textbf{169} (2007), no. 2, 427--449. \MR {2318562 (2008e:57028)}

\bibitem{Kaloti_Stein_fillings_of_planar_open_books}Amey Kaloti, \emph{Stein fillings of planar open books}, arXiv:1311.0208.

\bibitem{Li_Wang_support_genera_of_legendrian_knots}Youlin Li and Jiajun Wang, \emph{The support genera of certain Legendrian knots}, J. Knot Theory Ramifications \textbf{21} (2012), no. 11, 1250105, 8. \MR {2969635}

\bibitem{Lisca_Classification_of_Stein_fillings_on_lens_spaces}Paolo Lisca, On symplectic fillings of lens spaces, Trans. Amer. Math. Soc. \textbf{360} (2008), no. 2, 765--799
(electronic). \MR {2346471 (2008h:57039)}

\bibitem{Loi_Piergallini_Lefschetz_fibrations}Andrea Loi and Riccardo Piergallini, \emph{Compact Stein surfaces with boundary as branched covers of $B^4$},
Invent. Math. \textbf{143} (2001), no. 2, 325--348. \MR {1835390 (2002c:53139)}

\bibitem{McCommand_Margalit_Braid_Group}Dan Margalit and Jon McCammond, \emph{Geometric presentations for the pure braid group}, J. Knot Theory
Ramifications \textbf{18} (2009), no. 1, 1--20. \MR {2490001 (2010a:20086)}

\bibitem{McDuff_Rational_ruled_surfaces}Dusa McDuff, \emph{The structure of rational and ruled symplectic 4-manifolds}, J. Amer. Math. Soc. \textbf{3} (1990),
no. 3, 679--712. \MR {1049697 (91k:58042)}

\bibitem{Ohta_Ono_Classifications_of_Stein_fillings_1}Hiroshi Ohta and Kaoru Ono, \emph{Simple singularities and topology of symplectically filling 4-manifold}, Comment. Math. Helv. \textbf{74} (1999), no. 4, 575--590. \MR {1730658 (2001b:57060)}

\bibitem{Ohta_Ono_Classifications_of_Stein_fillings_2}Hiroshi Ohta and Kaoru Ono, \emph{Simple singularities and symplectic fillings}, J. Differential Geom. \textbf{69} (2005), no. 1, 1--42.
\MR {2169581 (2006e:53155)}


\bibitem{Plamenevskaya_VHMorris_Planar_open_books} Olga Plamenevskaya and Jeremy Van Horn-Morris, \emph{Planar open books, monodromy factorizations and
symplectic fillings}, Geom. Topol. \textbf{14} (2010), no. 4, 2077--2101. \MR {2740642 (2012c:57047)}

\bibitem{Schoenenberger_thesis_upenn} Stephan Sch{\"o}nenberger, \emph{Determining symplectic fillings from planar open books}, J. Symplectic Geom. \textbf{5}
(2007), no. 1, 19--41. \MR {2371183 (2008k:53196)}

\bibitem{Starkston_Symplectic_fillings_of_Seifert_fibered_spaces}Laura Starkston, \emph{Symplectic fillings of Seifert fibered spaces}, arXiv:1304.2420v2.

\bibitem{Stipsicz_gauge_theory_and_stein_fillings_of_certain_3_manifolds}Andr{\'a}s I. Stipsicz, \emph{Gauge theory and Stein fillings of certain 3-manifolds}, Turkish J. Math. \textbf{26} (2002),
no. 1, 115--130. \MR {1892805 (2003b:57038)}

\bibitem{Thurston_Wilkenkeper_Contact_structures_from_open_books} W. P. Thurston and H. E. Winkelnkemper, \emph{On the existence of contact forms}, Proc. Amer. Math. Soc. \textbf{52}
(1975), 345--347. \MR {0375366 (51 \#11561)}

\bibitem{Thurston_3_Manifolds_Kleinian_groups_and_hyperbolic_geometry} William P. Thurston, \emph{Three-dimensional manifolds, Kleinian groups and hyperbolic geometry}, Bull. Amer.
Math. Soc. (N.S.) \textbf{6} (1982), no. 3, 357--381. \MR {648524 (83h:57019)}

\bibitem{thurston_geometry_and_dynamics_of_diffeomorphism_of_surface}William P. Thurston, \emph{On the geometry and dynamics of diffeomorphisms of surfaces}, Bull. Amer. Math. Soc. (N.S.) \textbf{19}
(1988), no. 2, 417--431. \MR {956596 (89k:57023)}

\bibitem{Wendl_Planar_open_books} Chris Wendl, \emph{Strongly fillable contact manifolds and J-holomorphic foliations}, Duke Math. J. \textbf{151} (2010),
no. 3, 337--384. \MR {2605865 (2011e:53147)}


\end{thebibliography}
\def\cprime{$'$} \def\cprime{$'$}
\providecommand{\bysame}{\leavevmode\hbox to3em{\hrulefill}\thinspace}
\providecommand{\MR}{\relax\ifhmode\unskip\space\fi MR }
\providecommand{\MRhref}[2]{%
  \href{http://www.ams.org/mathscinet-getitem?mr=#1}{#2}
}
\providecommand{\href}[2]{#2}

\end{document}